\documentclass[a4paper,11pt]{amsart}

\usepackage{latexsym}
\usepackage{amssymb,amsmath, amsthm}
\usepackage{graphics}
\usepackage{url}
\usepackage[vcentermath, enableskew]{youngtabMW}
  \newcommand{\yngt}{\youngtabloid}

\newtheorem{theorem}{Theorem}[section] 
\newtheorem{corollary}[theorem]{Corollary}
\newtheorem{proposition}[theorem]{Proposition}

\newtheorem{lemma}[theorem]{Lemma}

\theoremstyle{definition}
\newtheorem{definition}[theorem]{Definition} 

\newtheorem{example}[theorem]{Example}

\newcommand{\N}{\mathbf{N}}

\renewcommand{\epsilon}{\varepsilon}
\renewcommand{\emptyset}{\varnothing}

\DeclareMathOperator{\Irr}{Irr}

\DeclareMathOperator{\Inf}{Inf}
\DeclareMathOperator{\Def}{Def}
\DeclareMathOperator{\Res}{Res}
\DeclareMathOperator{\Ind}{Ind}
\DeclareMathOperator{\Defres}{Def\hskip-0.3pt res}

\newcommand{\Ch}{\mathcal{C}}
\newcommand{\sgn}{\mathrm{sgn}}
\DeclareMathOperator{\sgnop}{sgn}
\newcommand{\nsgn}{\epsilon}

\newcommand{\wb}{\circ}
\newcommand{\bb}{\bullet}

\newcounter{thmlistcnt}
\newenvironment{thmlist}%
	{\setcounter{thmlistcnt}{0}%
	\begin{list}{\emph{(\roman{thmlistcnt})}}{%
		\usecounter{thmlistcnt}%
		\setlength{\topsep}{0pt}%
		\setlength{\leftmargin}{0pt}%
		\setlength{\itemsep}{0pt}%
		\setlength{\itemindent}{17pt}}%
	}%
	{\end{list}}%

\newcounter{deflistcnt}
	{\setcounter{deflistcnt}{0}%
	\begin{list}{(\roman{deflistcnt})}{%
		\usecounter{deflistcnt}%
		\setlength{\topsep}{0pt}%
		\setlength{\leftmargin}{0pt}%
		\setlength{\itemsep}{0pt}%
		\setlength{\itemindent}{17pt}}%
	}%
	{\end{list}}%

\newcommand{\ch}{\operatorname{ch}}
\newcommand{\lda}{\lambda}
\newcommand{\lan}{\langle}
\newcommand{\ran}{\rangle}

\renewcommand{\theta}{\vartheta}
\renewcommand{\epsilon}{\varepsilon}

\usepackage{caption}
\address{
Anton Evseev \\
School of Mathematics  \\
University of Birmingham \\
Edgbaston \\
Birmingham B15 2TT, UK  \\
}
\email{a.evseev@bham.ac.uk}

\address{
Rowena Paget \\
School of Mathematics, Statistics \& Actuarial Science  \\
University of Kent \\
Canterbury \\
Kent CT2 7NF, UK \\
}
\email{r.e.paget@kent.ac.uk}

\address{
 Mark Wildon \\
Mathematics Department \\
Royal Holloway, University of London \\
Egham TW20 0EX, UK \\
}
\email{mark.wildon@rhul.ac.uk}



\begin{document}

\title[Character deflations]{Character deflations
and a 
generalization of the Murnaghan--Nakayama rule}
\author{Anton Evseev, Rowena Paget and Mark Wildon}
\date{October 2013} 

\begin{abstract}
Given natural numbers $m$ and $n$, 
we define a deflation map from the 
 characters of the symmetric group
$S_{mn}$ to the characters of~$S_n$. 
This map is obtained by 
first restricting a character of $S_{mn}$ to the
wreath product $S_m \wr S_n$, and then taking the sum of 
the
 irreducible constituents
 of the restricted character
 on which the base group $S_m \times \cdots \times S_m$
acts trivially. 
We prove a combinatorial formula which gives the
values of the images of the irreducible characters of $S_{mn}$
under this map.
We also prove an
 analogous result for more general deflation maps
in which the base group is not required to act trivially.  
These results generalize 
the Murnaghan--Nakayama rule and special cases of the Littlewood--Richardson rule. 
As a corollary we obtain a new combinatorial formula
for the character multiplicities that are the subject of the
long-standing Foulkes' Conjecture.  
Using this formula we verify Foulkes' Conjecture in some new cases. 
\end{abstract}

\maketitle

\thispagestyle{empty}
\section{Introduction}\label{sec:intro}

Tableaux combinatorics is a pivotal theme  
in the representation theory of the symmetric groups.
Fundamental results in this area include the Murnaghan--Nakayama
rule for the values taken by irreducible characters of symmetric groups
and the Littlewood--Richardson rule (as well as its special case, Young's rule), 
which determines the restrictions of 
irreducible characters to Young subgroups of symmetric groups.

The two main results of this paper are Theorems~\ref{thm:cd} and~\ref{thm:hookdefl},
which give a combinatorial description of 
the restrictions of characters of symmetric groups to their
maximal imprimitive subgroups. Theorem~\ref{thm:cd} is a 
simultaneous generalization of the Murnaghan--Nakayama rule and Young's rule.
Theorem~\ref{thm:hookdefl} gives a further generalization, in which
Young's rule is replaced by a family of special cases of the Littlewood--Richardson rule.

As a corollary of Theorem~\ref{thm:cd}, we obtain in Proposition~\ref{prop:Foulkes}
a new recursive formula for the character multiplicities that are the subject of Foulkes' Conjecture, a long-standing problem which spans representation theory, invariant theory and algebraic combinatorics. We use this formula to verify
Foulkes' Conjecture in some new cases, extending the results in \cite{MN}.
 Figures~\ref{fig:FoulkesGraph} and~\ref{fig:FoulkesGraphsComp} in \S 5 show some of the data computed using this formula.

\subsection{Character deflations}
We now introduce the ideas needed
to state Theorem~\ref{thm:cd}.
By a construction originally due
to Frobenius,
the irreducible characters of the symmetric group $S_r$  are canonically
labelled by the partitions of $r$. As is usual, we write $\chi^\lambda$ for the irreducible
character labelled by the partition $\lambda$, and 
$\chi^{\lambda / \mu}$ for the character labelled by the
skew-partition $\lambda / \mu$. We refer the reader to 
\cite[Chapter 2]{JK} or \cite[\S 7.18]{StanleyII} for a construction
of these characters and to \cite[page 309]{StanleyII}
for background on skew-partitions.

For each $r \in \mathbf{N}$, it is well known (see,
for example, \cite[Exercise 5.2.8]{DixonMortimer}) that the maximal 
imprimitive subgroups of $S_r$ are precisely the imprimitive wreath 
products $S_m \wr S_n \le S_{r}$ for $m$, $n \in \mathbf{N}$ such that $mn = r$. 
Let $\theta$ be a character of $S_m$, and let $V$ be a representation of $S_m$ affording $\theta$. Then $V^{\otimes n}$ 
is a representation of 
the base group $S_m\times \cdots \times S_m$. The 
complement~$S_n$ of this base group acts on $V^{\otimes n}$ by permuting the factors:
\[
 g(v_1\otimes \cdots \otimes v_n) = v_{g^{-1} (1)} \otimes \cdots \otimes v_{g^{-1} (n)} 
\]
for $g \in S_n$ and $v_1, \ldots, v_n \in V$.
These two actions combine to give a representation of $S_m\wr S_n$ on $V^{\otimes n}$ (see~\cite[4.3.6]{JK}). 
 We shall denote by $\widetilde{\theta^{\times n}}$ the character of $S_m\wr S_n$ 
afforded by this representation.
We also need the characters of $S_m \wr S_n$ whose
kernel contains \hbox{$S_m \times \cdots \times S_m$}.
These characters are precisely the inflations of the
characters of $S_n$ to $S_m \wr S_n$ along the canonical surjection
\hbox{$S_m \wr S_n \twoheadrightarrow S_n$}. 
If $\nu$ is a partition of $n$,
we denote by $\Inf_{S_n}^{S_m \wr S_n} \!\chi^\nu$ the irreducible
character of $S_m \wr S_n$ constructed in this way. 
It is easily seen that the
characters $\widetilde{\theta^{\times n}} \Inf_{S_n}^{S_m \wr S_n}\!\chi^\nu$ 
obtained
by multiplying characters of these two types are irreducible.
(By \cite[Theorem 4.3.33]{JK}, any irreducible character of $S_m \wr S_n$ is induced from a suitable product of characters of this form.)

Given a finite group $G$, we let
$\Ch(G)$ denote the abelian group of virtual characters of $G$.

\begin{definition}\label{def:def}
Let $m$, $n \in \mathbf{N}$ and let $\theta$ be an irreducible
character of $S_m$. Let~$\xi$ be an irreducible character
of $S_m \wr S_n$. We define
\[ \Def^\theta_{S_n} \xi = \begin{cases}
\chi^\nu & \text{if $\xi = \widetilde{\theta^{\times n}} \Inf_{S_n}^{S_m
\wr S_n}\!\chi^\nu$
where $\nu$ is a partition of $n$} \\
0 & \text{otherwise}. \end{cases} \]
Let
$\Def_{S_n}^\theta : \Ch(S_m \wr S_n) \rightarrow \Ch(S_n)$ be the
group homomorphism defined by linear extension of this definition. 
Given $\psi \in \Ch(S_m \wr S_n)$,
we say that $\Def_{S_n}^\theta \psi$ is the \emph{deflation of $\psi$
with respect to $\theta$}.
Let $\Defres_{S_n}^\theta : \Ch(S_{mn}) \rightarrow
\Ch(S_n)$ be the group homomorphism defined by
\[ \Defres^\theta_{S_n} \chi = \Def^\theta_{S_n} 
\Res^{S_{mn}}_{S_m \wr S_n} \chi \]
for $\chi \in \Ch(S_{mn})$.
\end{definition}
 In the case when $\theta$ is the trivial character of $S_m$, we shall omit $\theta$ and simply write $\Def_{S_n}$ and $\Defres_{S_n}$. 
If $V$ is a complex representation of $S_m \wr S_n$ with
 character~$\chi$, then $\Def_{S_n}\chi$ is the
character of the maximal subrepresentation of $V$
on which the base group $S_m \times \cdots \times S_m$ acts
trivially. 

Theorem~\ref{thm:cd} 
gives a combinatorial rule for the values
of $\Defres_{S_n} \chi^{\lambda / \mu}$ where $\lambda / \mu$ is a 
skew-partition of $mn$. In order to state this rule, we review and extend the definition of a border-strip tableaux
(see~\cite[\S 7.17]{StanleyII}). 

Recall that a skew-partition $\sigma / \tau$ is said to be
a \emph{border strip} (or rim hook) if the Young diagram of $\sigma / \tau$
is connected and contains no $2 \times 2$ square. 
The $\emph{length}$ of the border strip $\sigma / \tau$
is $|\sigma / \tau|$ and its \emph{height} 
is defined to be one less than its number of non-empty rows. 
If $\lambda / \mu$ is a skew-partition, then
we define
a \emph{border-strip tableau} of shape $\lambda / \mu$
to be an assignment of the elements of a set $J \subseteq \N$ 
to the boxes of the Young
diagram of $\lambda / \mu$ so that the rows and columns
are non-decreasing, and for each $j \in J$, the boxes labelled $j$
form a border strip; if $J = \{1,\ldots,k\}$, and for each $j \in J$
the border strip formed by the boxes labelled $j$ has length~$\alpha_j$,
then we say that the 
tableau has \emph{type} $(\alpha_1,\ldots,\alpha_k)$.
We need
the following three further definitions, which are illustrated in
Figure~\ref{fig:border_strip} and Example~\ref{example:border_strip} below.

\begin{definition}\label{def:sgn}
Let $T$ be a border-strip tableau. The
\emph{sign} of $T$ is defined
by $\sgn(T) = (-1)^h$, where $h$ is the sum of the 
heights of the border strips forming $T$.
\end{definition}

\enlargethispage{3.5pt}

\begin{figure}[b]
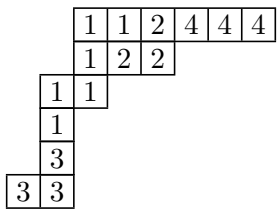
 
\[ \young(::112444,::122,:11,:1,:3,33) \]
\caption{A border-strip tableau of shape $(8,5,3,2,2,2) / (2,2,1,1,1)$ and type $(6,3,3,3)$. The heights of the border strips labelled $1,2,3,4$ are $3,1,1,0$ respectively, 
and the sign of this border-strip tableau is thus $-1$. }      
\label{fig:border_strip}         
\end{figure}

\begin{definition}\label{def:row_number}
Let $\lambda / \tau$ be a border strip in a partition $\lambda$.
If the 
lowest-numbered row of $\lambda$ met by $\lambda / \tau$
is row $k$ then we define the \emph{row number} of $\lambda / \tau$
to be $k$, and write $N(\lambda/ \tau) = k$.
\end{definition}

\medskip
Note that if $T$ is a border-strip tableau of
shape $\lambda / \mu$ and type $(\alpha_1, \ldots, \alpha_k)$
then there are partitions
\[ \mu = 
\lambda^0 \subset \lambda^1 \subset \cdots \subset \lambda^{k-1}
\subset \lambda^{k} = \lambda \]
such that for each $j \in \{1,\ldots,k\}$,
the border strip in $T$ labelled $j$ is $\lambda^j /\lambda^{j-1}$. 

\begin{definition}\label{def:a}
Let $m$, $n \in \N$ and let $\lambda / \mu$ be a skew-partition of $mn$.
Given a composition $\gamma = (\gamma_1, \ldots, \gamma_d)$ of $n$,
let $\gamma^{\star m} = (\gamma_1, \ldots, \gamma_1,
\ldots, \gamma_d, \ldots, \gamma_d)$ denote the composition of $mn$ obtained from 
$\gamma$ by
repeating each part $m$ times.
An \emph{$m$-border-strip tableau} of \emph{shape $\lambda / \mu$} and 
\emph{type~$\gamma$} is a border-strip tableau of shape $\lambda / \mu$
and type~$\gamma^{\star m}$ such that for each $j \in \{1,2,\ldots,d\}$,
the row numbers of the border strips
\[ \lambda^{(j-1)m+1} / \lambda^{(j-1)m},  
\ldots ,
\lambda^{jm} / \lambda^{jm-1} \]

\vskip-4pt
\noindent corresponding to the  $m$ parts in $\gamma^{\star m}$
of length
$\gamma_j$ satisfy
\begin{equation}\label{eq:row_numbers} 
N(\lambda^{(j-1)m+1} / \lambda^{jm}) 
\ge \cdots \ge N(\lambda^{jm} / \lambda^{jm-1}). 
\end{equation}
Let
\[ a_{\lambda/ \mu, \gamma} = \sum_{T} \sgn(T) \] 
where the sum is over all $m$-border-strip tableaux $T$ of shape~$\lambda / \mu$
and type~$\gamma$.
\end{definition}

\begin{theorem}\label{thm:cd} 
Let $m$, $n \in \N$ and let $\lambda / \mu$ be a skew-partition of $mn$.
If $\gamma$ is a composition of $n$ and $g \in S_n$ has cycle type $\gamma$
then
\[ (\Defres_{S_n} \chi^{\lambda / \mu})(g) = a_{\lambda/\mu,\gamma}. \]
\end{theorem}

\begin{example}\label{example:border_strip}
Let $\lambda = (6,5,3,2)$ and let $\mu = (3,1)$.
The three different $2$-border-strip tableaux of shape 
$\lambda / \mu$ and type $\gamma = (1,2,3)$ are shown below.
\[
\rule[-27pt]{0pt}{29pt}\young(:::266,:1556,345,34) 
\qquad
\young(:::466,:1246,355,35) 
\qquad
\young(:::266,:1446,355,35)
\]
As required by Definition~\ref{def:a}, for each $j \in \{1,2,3\}$, the
row number of the border strip labelled $2j-1$ in each tableau
is at least the row number of the border strip labelled $2j$. Thus  
the first border strip corresponding to each part of $\gamma$
is added no higher up in each partition diagram than the second.
The sums of the heights of the border strips forming these tableaux are $4$, $4$ and $3$
and so their signs are $+1$, $+1$ and $-1$, 
respectively. By
Definition~\ref{def:a} we have~$a_{\lambda / \mu,\gamma} = 1$. Hence
Theorem~\ref{thm:cd} implies that
$(\Defres_{S_6}\chi^{(6,5,3,2) / (3,1)}) (g) = 1$
if  $g \in S_6$
has cycle type $(1,2,3)$. 
\end{example}

Deflation is closely related to plethysm of Schur functions (see~\cite[\S I.8]{Macdonald1995}). 
In fact, using the standard correspondence between characters of symmetric groups and symmetric functions, 
one can show that the special case $\gamma=(n)$ of Theorem~\ref{thm:cd} is equivalent to a result proved in~\cite[Section 9]{DLT}.
Also, in the special case $\mu=\varnothing$, the combinatorial description of Theorem~\ref{thm:cd} can be shown (using our Lemma~\ref{lemma:horiz} below)
 to be equivalent to the one given for plethysm by Macdonald~\cite[\S I.8, Example 8]{Macdonald1995}. These connections are discussed in more detail in~\S\ref{sec:concl}.

We prove Theorem~\ref{thm:cd} in \S\S\ref{sec:averaging}--\ref{sec:proof}.
The only prerequisites, apart from some basic character theory, are the
Murnaghan--Nakayama rule and the combinatorics of the abacus.  
In addition to being self-contained, our proof is highly combinatorial in the sense that the key steps, given in \S 3, can all be
stated in terms of explicit bijections between certain classes of tableaux.

\subsection{Some special cases}
It is clear than if $m=1$ then
$\Defres_{S_n} \chi=
\chi$ for any character $\chi$ of $S_n$, and so the 
special case $m=1$ of Theorem~\ref{thm:cd} asserts that
$\chi^{\lambda / \mu} (g) = a_{\lambda/ \mu , \gamma}$
for any skew-partition $\lambda / \mu$ of $n$ and
any element $g \in S_n$ of cycle type $\gamma$.
Equivalently, 
\[ \chi^{\lambda / \mu} (g) = \sum_{T} \sgn(T) \]
where the sum is over all border-strip tableaux of shape $\lambda / \mu$ 
and type $\gamma$. This is 
the Murnaghan--Nakayama rule, as stated in
\cite[Equation~(7.75)]{StanleyII}. 
It should be noted that we require the Murnaghan--Nakayama rule in 
\S\ref{subsec:omega_proof} below, and
so our work does not provide a new proof of this result.
In practice the Murnaghan--Nakayama rule is most frequently
used as a recursive formula for the values of characters or skew characters. 
Equation~\eqref{eq:defres_decomp_sgn} at the end of \S\ref{sec:proof}
formulates
Theorem~\ref{thm:cd} in this way.

As Stanley observes in \cite[page 348]{StanleyII}, 
it is far from obvious that
the character values given by the Murnaghan--Nakayama
rule applied to a skew-partition $\lambda / \mu$ and a composition
$\gamma$ are independent of the order of the parts of $\gamma$.
This remark applies even more strongly to Theorem~\ref{thm:cd}.
For example,
the reader may check that if $\lambda / \mu = (6,5,3,2) / (3,1)$,
as in Example~\ref{example:border_strip}, and $\gamma' = (2,1,3)$,
then there is a unique $2$-border-strip tableau of shape
$(6,5,3,2) / (3,1)$ and type $\gamma'$. Thus $a_{\lambda/\mu,\gamma'} = 1$,
but the sums defining  $a_{\lambda/ \mu,\gamma}$
and $a_{\lambda / \mu, \gamma'}$ are different.

Another 
special case of Theorem~\ref{thm:cd} worth noting
occurs when $g$ is the identity
element of $S_n$. If $\xi$ is an irreducible character of 
$S_m \wr S_n$ then either the base group 
$B = S_m \times \cdots \times S_m$ is contained in
the kernel of $\xi$ 
and $\left<\Res_B \xi, 1_B\right>
= \xi(1)$, or  
$\left<\Res_B \xi, 1_B\right> = 0$.
Hence, by linearity, we have 
\begin{equation}
\label{eq:deflation_special} 
(\Defres_{S_n} \chi)(1) = \left<\Res_B \chi ,1_B\right> 
\end{equation}
for any character $\chi$ of $S_{mn}$.
It now follows from Theorem~\ref{thm:cd} and Frobenius reciprocity that
\[ 
 a_{\lambda / \mu, (1^n)} =
\left< \chi^{\lambda / \mu},
 \Ind_{S_m \times \cdots \times S_m}^{S_{mn}}
1_{S_m} \times \cdots \times 1_{S_m} \right> 
\]
for any skew-partition $\lambda / \mu$ of $mn$.
It is clear from Definition~\ref{def:a} that $a_{\lambda / \mu, (1^n)}$ is the
number of semi-standard tableaux of shape $\lambda / \mu$ and
type $(m^n)$. Therefore, by setting $\mu = \emptyset$ in the previous equation, 
we obtain
a special case of Young's rule (see 
\cite[2.8.5]{JK} or
\cite[Proposition 7.18.7]{StanleyII}).

\subsection{Outline of the paper}
The remainder of the paper proceeds as follows. 
Throughout, we shall
 adopt the convention that if $\alpha$
is a partition of $r \in \N$, then $g_\alpha \in S_r$ is
an element of
cycle type~$\alpha$, and $z_\alpha$ is
the size of the centralizer of $g_\alpha$ in $S_r$. 
(The choice of $g_\alpha$
within the
conjugacy class is 
irrelevant.) If $\alpha = (\alpha_1,\ldots,\alpha_k)$, we write $n\alpha = (n\alpha_1,\ldots,n\alpha_k)$. 

In \S\ref{sec:averaging} we prove
 Proposition~\ref{prop:deflation},
 which implies that if  $\chi$ is a character of~$S_{mn}$
 and $g\in S_n$,
then $(\Defres_{S_n} \chi)(g)$ is the average value of
 $\chi$ on the coset of the
base group $S_m \times \cdots \times S_m$ in $S_m \wr S_n$
corresponding to $g$. 
Equation~\eqref{eq:deflation_special} above
is a special case of this result. 
 In the case when $g \in S_n$ is an $n$-cycle, 
 we obtain Proposition~\ref{prop:omega}(ii), which implies that if  $\lambda / \mu$ is a skew-partition of $mn$, then
\begin{equation}
\label{eq:cycle_deflation}
 \Defres_{S_n} \chi^{\lambda / \mu}(g) = 
\sum_{\alpha}\frac{\chi^{\lambda / \mu}(g_{n\alpha})}{z_\alpha} 
\end{equation}
where the sum is over all partitions $\alpha$ of $m$, and $n\alpha$ denotes the partition obtained from $\alpha$ by multiplying each of its parts by $n$.

In \S\ref{sec:skew} we state a theorem of Farahat (see~\cite[Section 4]{Farahat1954}), which gives a formula for the character
values $\chi^{\lambda/\mu}(g_{n\alpha})$ appearing on the right-hand side of~\eqref{eq:cycle_deflation}. We 
then give a character-theoretic proof of this theorem. 

In \S\ref{sec:proof} we combine the results of \S 3 and \S 4 to
 show that Theorem~\ref{thm:cd} holds when $g \in S_n$ is an $n$-cycle (Proposition~\ref{prop:basecase}) and then to deduce it in general. 

In \S\ref{sec:Foulkes} we apply our results on deflations to Foulkes' Conjecture 
on permutation characters the symmetric group. In particular, we prove a new recursive formula for the 
character multiplicities that appear in this conjecture. Using this formula we check Foulkes' Conjecture
in some new cases, extending the results in \cite{MN}.

In \S\ref{sec:gendef}, we consider the more general
deflation maps $\Def^\theta_{S_n}$. 
When $\theta=\chi^{(a,1^b)}$ is labelled by a hook partition, Theorem~\ref{thm:hookdefl} gives a combinatorial description of the value of
$\Defres^{\theta}_{S_n} \chi^{\lambda/ \mu}$ on an $n$-cycle.
This result generalizes the case $\gamma=(n)$ of Theorem~\ref{thm:cd} and may be viewed as
a simultaneous generalization of the 
Murnaghan--Nakayama rule
and a special case of the  Littlewood--Richardson rule.
 We  also give an illustrative example showing how our
 methods can be used to compute values of deflated characters in the non-hook case.  

Finally, in \S\ref{sec:concl}, we discuss the aforementioned connections between Theorem~\ref{thm:cd} and results in~\cite{DLT,Macdonald1995} stated in terms of symmetric functions.

\section{Deflation by averaging}\label{sec:averaging}

Let $m$, $n \in \N$. 
We shall think
of $S_m \wr S_n$ as the group of permutations of
\[ \{1,\ldots, m\} \times \{1,\ldots, n\} \]
that leaves invariant the set of blocks of the form $\Delta_j = \{(1,j), \ldots, (m,j)\}$, 
$1 \le j \le n$.					
Given $h_1, \ldots, h_n \in S_m$ and $g \in S_n$, we
write $(h_1,\ldots,h_n;g)$
for the permutation which sends $(i,j)$ 
to $(h_{gj} i,gj)$. 
This left action is equivalent to the action
defined in \hbox{\cite[4.1.18]{JK}}. 
Let $B = S_m \times \cdots \times S_m$ denote
the base group in the wreath product. As shorthand,
if $k = (h_1,\ldots,h_n) \in B$ 
then we shall write $(k\,;g)$ for $(h_1,\ldots,h_n;g)$.

\begin{lemma}\label{lemma:sum}
Let $m$, $n \in \mathbf{N}$, let $\theta$ be an irreducible character of $S_m$, 
and let $\xi$ be an irreducible character of $S_m \wr S_n$. 
 If $\xi = \widetilde{\theta^{\times n}} \Inf_{S_n}^{S_m \wr S_n}\!\chi^\nu$ for some partition $\nu$ of $n$ then 
 \[ \frac{1}{|B|} \sum_{k \in B} \xi (k\, ;g) 
 \widetilde{\theta^{\times n}} (k\, ;g) = \chi^\nu(g), \]
and if $\xi\in \Irr(S_m \wr S_n)$ is not of this form then the left-hand side is zero.
\end{lemma}

\begin{proof}
Suppose that the left-hand side is non-zero. The character
$\widetilde{\theta^{\times n}}$ of $S_m \wr S_n$ restricts to
the irreducible character $\theta \times \cdots \times \theta$ of $B$.
Hence,
by  \cite[Lemma 8.14(b)]{IsaacsChars}, applied with $G = S_m \wr S_n$
and $N = B$, we have
$\left<\Res_B \xi, \Res_B \widetilde{\theta^{\times n}}\right> \not=0$.
It follows by Frobenius Reciprocity that $\xi$ is a constituent of
\[ \Ind_{B}^{S_m \wr S_n} (\theta \times \cdots \times \theta) = 
\sum_{\nu} \chi^\nu(1) \widetilde{\theta^{\times n}} \Inf_{S_n}^{S_m \wr S_n}\!\chi^\nu ,\]
where the sum is over all partitions $\nu$ of $n$.
Since $\xi$ is irreducible 
we must have $\xi =  \widetilde{\theta^{\times n}}\Inf_{S_n}^{S_m \wr S_n}\!\chi^\nu$ for some 
$\nu$. Therefore the left-hand side in the lemma is  
\[ \frac{\chi^\nu(g)}{|B|} \sum_{k \in B}
\bigl( \widetilde{\theta^{\times n}}(k\, ;g) \bigr)^2 \]
which is equal to $\chi^\nu(g)$ by \cite[Lemma 8.14(c)]{IsaacsChars}.
\end{proof}

By Definition~\ref{def:def}, 
we have $\Defres^\theta_{S_n}
(\widetilde{\theta^{\times n}} \Inf_{S_n}^{S_m \wr S_n}\!\chi^\nu)(g) = 
\chi^\nu(g)$ for all $g\in S_n$. The next proposition therefore
follows immediately from Lemma~\ref{lemma:sum}. 

\begin{proposition}\label{prop:deflation}
Let $m$, $n \in \N$, let $\theta$ be an irreducible
character of $S_m$, and let $\psi$ be a character of $S_m \wr S_n$.
If $g \in S_n$ then
\begin{flalign*}
    &&(\Def^\theta_{S_n} \psi)(g) = 
\frac{1}{|B|} \sum_{k \in B} 
\psi (k\, ;g) \widetilde{\theta^{\times n}} (k\,;g). && \qed  
\end{flalign*}
\end{proposition}

\begin{corollary}\label{cor:deflation}
Let $m$, $n \in \N$, let $\theta$ be an irreducible
character of $S_m$, and let $\psi$ be a character of $S_{m} \wr S_n$.
If  $g \in S_n$ is an $n$-cycle then
\[ (\Def^\theta_{S_n} \psi)(g) =
\frac{1}{m!} \sum_{h \in S_m} \psi (h,1,\ldots,1;g)\theta(h).\]
\end{corollary}

\begin{proof}
Suppose that $g$ is the $n$-cycle $(x_1 \, x_2 \, \ldots \, x_n)$. 
By~\cite[4.2.8]{JK}, the permutations $(h_1,\ldots,h_n;g)$ and
$(h_1',\ldots,h_n';g) \in S_m \wr S_n$ are conjugate
in $S_m \wr S_n$ if and only if 
the elements $h_{x_n} h_{x_{n-1}}\ldots h_{x_1}$ and   $h'_{x_n} h'_{x_{n-1}}\ldots h'_{x_1}$ are conjugate in $S_m$.
In particular, each conjugacy class of $S_m \wr S_n$ 
which meets $\{(k\,; g) : k \in B\}$ has 
a representative of the form $(h,1,\ldots,1;g)$. Moreover,
the number of elements $(h_1,h_2,\ldots,h_n;g)$ conjugate
to $(h,1,\ldots,1;g)$ is
$m!^{n-1} |h^{S_m}|$,
since $h_2, \ldots, h_n$ may be chosen arbitrarily, 
and then 
$h_1$ must be chosen so that 
$h_{x_n} h_{x_{n-1}}\cdots h_{x_1} \in h^{S_m}$.
It follows that
\begin{align*}
\sum_{k \in B} \psi(k\, ;g) \widetilde{\theta^{\times n}}(k\, ;g)  &= 
m!^{n-1} \sum_{h \in S_m}
 \psi (h,1,\ldots,1;g) \widetilde{\theta^{\times n}}(h,1,\ldots
,1;g) \\
&= m!^{n-1} \sum_{h \in S_m} \psi (h,1,\ldots,1;g) \theta(h)
\end{align*}
where the second equality uses~\cite[Lemma 4.3.9]{JK}. 
Now apply Proposition~\ref{prop:deflation} to the left-hand side.
\end{proof}

The following definition and lemma allow for a more convenient
statement of Corollary~\ref{cor:deflation}.

\begin{definition}\label{def:omega}
Let $m$, $n \in \N$, let $g \in S_n$ be
an $n$-cycle, and
let $\psi$ be a character of $S_m \wr S_n$. We define
$\omega(\psi)$ to be the 
 class function on $S_m$ such that
\[ \omega(\psi) (h) = \psi(h,1,\ldots,1;g)  \]
for all $h \in S_m$.
\end{definition}

\begin{lemma}\label{lemma:cycles}
Let $m$, $n \in \N$. If $g\in S_n$ is an $n$-cycle
and $h \in S_m$ has cycle
type $\alpha$ 
then $(h,1,\ldots,1;g) \in S_m \wr S_n$ has cycle
type $n\alpha$.
\end{lemma}

\begin{proof}
It suffices to show that if $\mathcal{O}$ is an orbit
of $h$ on $\{1,2,\ldots,m\}$ then $\mathcal{O} \times \{1,\ldots,n\}$
is an orbit of $(h,1,\ldots,1;g)$ in its action on 
$\{1,\ldots,m\} \times \{1,\ldots,n\}$.
We leave this to the reader as an easy exercise. 
\end{proof}

The next proposition follows easily from
 Lemma~\ref{lemma:cycles} and
Corollary~\ref{cor:deflation}.

\vbox{
\begin{proposition}\label{prop:omega}
Let $m$, $n \in \N$, and let $\chi$ be a character of $S_{mn}$. 

\begin{thmlist}
\item
If
$\alpha$ is a partition of $m$ then
\[ \omega(\Res_{S_m \wr S_n} \chi) (g_\alpha) = \chi(g_{n\alpha}).\]
\item If $\theta$ is an irreducible
character of $S_m$ and $g \in S_n$ is an $n$-cycle then
\begin{align*} (\Defres^\theta_{S_n} \chi)(g) &= 
\left< \omega(\Res_{S_m \wr S_n} \chi), \theta  \right> \\
&= \sum_{\alpha} 
\frac{\chi(g_{n\alpha})}{z_\alpha} \theta(g_\alpha) \end{align*}
where the sum is over all partitions $\alpha$ of $m$.\hfill$\qed$
\end{thmlist}
\end{proposition}}

The character
value $\chi(g_{n\alpha})$ in part (i) is the subject of
Theorem~\ref{thm:ind}; we shall see that
combining this theorem with part (ii) gives
Equation~\eqref{eq:omegares} in \S\ref{sec:proof} below. 
Note also that part~(ii) of the proposition
implies Equation~\eqref{eq:cycle_deflation} in~\S 2.

\section{Skew characters}
\label{sec:skew}

In this section we state and prove a result on the values of skew characters on elements of the form $g_{n\alpha}$ (Theorem~\ref{thm:ind}).
First, we give the necessary combinatorial definitions. In several arguments we shall refer to James' abacus notation for partitions, as described in \cite[page 78]{JK}.

\subsection{Quotients of skew-partitions}
\label{subsec:quotients}

We shall define $n$-quotients and $n$-signs
for the following class of skew-partitions.

\begin{definition}\label{def:n_dec}
Let $m$, $n \in \N$. We say that a skew-partition $\lambda /\mu$ of $mn$
is \emph{$n$-decomposable} if there
exists a border-strip tableau of shape $\lambda / \mu$   
and type~$(n^m)$.
\end{definition}

\begin{definition}\label{def:n_quo}
Let $m$, $n \in \N$ and let $\lambda / \mu$ be an $n$-decomposable
skew-partition of $mn$. Let $\Gamma(\lambda)$ be an abacus display
for $\lambda$ on an $n$-runner abacus
using $tn$ beads for some $t \in \N$. Let
$\Gamma(\mu)$ be the abacus display for $\mu$
obtained by performing an appropriate sequence of $m$ upward
bead moves on $\Gamma(\lambda)$. 
 (This is possible since $\lambda / \mu$ is $n$-decomposable.)
Let
$(\lambda^{(0)}, \ldots, \lambda^{(n-1)})$ and
$(\mu^{(0)}, \ldots, \mu^{(n-1)})$ be the $n$-quotients
of $\lambda$ and $\mu$ corresponding to
$\Gamma(\lambda)$ and $\Gamma(\mu)$, respectively.
The \emph{$n$-quotient} of $\lambda / \mu$ is defined to be
\[ (\lambda^{(0)} / \mu^{(0)}, \ldots, \lambda^{(n-1)} / \mu^{(n-1)}). \]
We define the $n$-\emph{sign}
of $\lambda / \mu$, denoted $\nsgn_n(\lambda/\mu)$ to be the sign of any border-strip tableau
of shape $\lambda / \mu$ and type $(n^m)$. 
\end{definition}

To avoid cumbersome restatements,
we adopt the convention that $\lambda^{(i)} / \mu^{(i)}$ always
has the meaning of Definition~\ref{def:n_quo} above.		
It is clear from the abacus that $\mu^{(i)}$ is a subpartition
of $\lambda^{(i)}$ for each $i \in \{0,\ldots,n-1\}$, and
so the $n$-quotient is well-defined. It follows
from Proposition~3.13 in~\cite{OlssonCombinatorics}, or our
Proposition~\ref{prop:nbij} below, that
the $n$-sign of a skew-partition is well defined. See \S\ref{subsec:example}
below for an example of these definitions and all the results in this section.

We remark that it appears to be impossible
to define the $n$-core of an arbitrary skew-partition.
The example $\lambda / \mu = (2,2) / (1)$
and $n=2$ illustrates the obstacles that arise.
Representing $\lambda$ on a $2$-runner abacus as
\[ \begin{matrix} \wb & \wb \\ \bb & \bb \end{matrix} \]
we see that either bead may be moved up,
giving two different skew-partitions from
which no border strip of length $2$ can be removed, namely $(2) /(1)$ 
and $(1,1) / (1)$. 
The $2$-quotients corresponding to these bead moves,
namely $((1),\emptyset)$ and
$(\emptyset,(1))$, are also different.

\begin{theorem}[Farahat]\label{thm:ind}
Let $m$, $n \in \mathbf{N}$ and let $\lambda / \mu$ be
a skew-partition of $mn$. Let $\alpha$ be a partition of $m$.
If~$\lambda / \mu$ is not
$n$-decomposable then $\chi^{\lambda / \mu}(g_{n \alpha}) = 0$.
If~$\lambda / \mu$ is $n$-decomposable and
$(\lambda^{(0)} / \mu^{(0)}, \ldots, \lambda^{(n-1)} / \mu^{(n-1)})$ is its $n$-quotient, 
then
\[
\chi^{\lambda / \mu}(g_{n\alpha}) = \nsgn_n(\lambda / \mu)
\Ind^{S_m}_{S_{\ell_0} \times \cdots \times S_{\ell_{n-1}}}
\bigl(
\chi^{\lambda^{(0)} / \mu^{(0)}} \times \cdots \times 
\chi^{\lambda^{(n-1)}/ \mu^{(n-1)}} \bigr)
(g_\alpha) \]
where $|\lambda^{(i)} / \mu^{(i)}| = \ell_i$. 
\end{theorem}

This result, stated in the alternative language of star diagrams, was first proved in \cite[Section~4]{Farahat1954}. The special case where $\mu$ is the 
$n$-core of $\lambda$
also follows from the correction by Thrall and Robinson \cite{ThrallRobinson1951} to Section 7 of
Robinson \cite{Robinson1947} or, alternatively, from Littlewood's result in~\cite[Section 2]{Littlewood1951}. 
Farahat's proof depends on an algebraic argument using Schur functions. 
A character-theoretic proof is given by Kerber, S{\"a}nger, and Wagner: see~\cite[Equation 3.6]{KSW}. 
In the remainder of this section, we give a shorter character-theoretic proof of Theorem~\ref{thm:ind}, expressing each side of the theorem as a sum, 
and then constructing a bijection between the summands. An example to illustrate this bijection is given in \S \ref{subsec:example}. 
(Our bijection is similar to that defined using \emph{Brettspiele} in~\cite{KSW}.)

\medskip
\enlargethispage{-6pt}

\subsection{A model for induction from a Young subgroup}
\label{subsec:model}

The following general result on the values of a character induced from a 
Young subgroup will be used in the proof of Theorem~\ref{thm:ind}.
(The notation is chosen to be consistent with this later use.)

\begin{lemma}\label{lemma:model}
Let $(\ell_0,\ldots,\ell_{n-1})$ be a composition of $m \in \N$. 
For each $i \in \{0,\ldots,n-1\}$, let
$\theta_i$ be a character of $S_{\ell_i}$. If $g\in S_m$
then
\[ \Ind^{S_m}_{S_{\ell_0} \times \cdots \times S_{\ell_{n-1}}} 
(\theta_0 \times \cdots \times \theta_{n-1})(g)
=
\sum_{\mathbf{t}} \theta_0(g_{\alpha_0(\mathbf{t})}) \ldots
\theta_{n-1}(g_{\alpha_{n-1}(\mathbf{t})}) 
\]
where the sum is over all $(\ell_0,\ldots,\ell_{n-1})$-tabloids 
$\mathbf{t}$ such that
$g \mathbf{t} = \mathbf{t}$, and $\alpha_i(\mathbf{t})$ is the cycle type of
the permutation induced by $g$ on
the entries of
 row $i+1$ of $\mathbf{t}$.
\end{lemma}

\begin{proof}
Let $\mathbf{t}_1$, \ldots, $\mathbf{t}_N$ 
be the $(\ell_0,\ldots,\ell_{n-1})$-tabloids. Let $\mathbf{s}$ be 
a $(\ell_0,\ldots,\ell_{n-1})$-tabloid fixed by the Young subgroup
$S_{\ell_0} \times \cdots \times S_{\ell_{n-1}}$.
For each $j$ such that $1 \le j \le N$, choose $x_j \in S_m$
such that $\mathbf{t}_j =  x_j \mathbf{s}$. Let
$\theta = \theta_0 \times \cdots \times \theta_{n-1}$. 
For each $g\in S_m$ we have
\[ \bigl( 
\Ind^{S_m}_{S_{\ell_0} \times \cdots \times 
S_{\ell_{n-1}}} \theta \bigr)
(g)
= 
\sum_j \theta (x_j^{-1}gx_j) \] 
where the sum is over all $j$ such that
\[ x_j^{-1}gx_j \in S_{\ell_0} \times \cdots \times S_{\ell_{n-1}}, \] 
or, equivalently, over all $j$ such
that $g \mathbf{t}_j  = \mathbf{t}_j$. If
$\Delta_1, \ldots, \Delta_q$ 
are the orbits of~$g$ on row $i+1$ 
of  $\mathbf{t}_j$, then
$x_j^{-1}\Delta_1, \ldots, x_j^{-1}\Delta_q$ 
are the orbits of $x_j^{-1}gx_j$ on 
row $i+1$ of 
$\mathbf{s}$.
Hence $x_j^{-1}gx_j$ acts with cycle type $\alpha_i(\mathbf{t}_j)$
on row $i+1$ of $\mathbf{s}$ and so
\[ \theta(x_j^{-1}gx_j) = \theta_0(g_{\alpha_0(\mathbf{t_j})})
\ldots \theta_{n-1}(g_{\alpha_{n-1}(\mathbf{t_j})}).
\]
The lemma follows.
\end{proof}

\subsection{Proof of Theorem~\ref{thm:ind}}\label{subsec:cfs}
\label{subsec:omega_proof}

Let $m$, $n \in \N$ and let $\lambda / \mu$ be a skew-partition
of~$mn$. Let $\alpha$ be a partition of $m$. 
If there is a 
border-strip tableau of shape $\lambda /\mu$ and type $n\alpha$
then it is clear from the abacus that  $\lambda / \mu$ is $n$-decomposable. 
Hence if $\lambda / \mu$
is not $n$-decomposable then, by the Murnaghan--Nakayama rule,
$\chi^{\lambda/\mu}(g_{n \alpha}) = 0$. 

We may therefore assume that $\lambda / \mu$ is $n$-decomposable.
Let $\ell_i = |\lambda^{(i)} / \mu^{(i)}|$ for
each $i \in \{0,\ldots,n-1\}$ and let
$H = {S_{\ell_0} \times \cdots \times S_{\ell_{n-1}}}$.
To show that
\begin{equation} \label{eq:charmain}
 \chi^{\lambda / \mu}(g_{n\alpha}) = \nsgn_n(\lambda / \mu)
\Ind^{S_m}_H
\bigl(
\chi^{\lambda^{(0)} / \mu^{(0)}} \times \cdots 
\times \chi^{\lambda^{(n-1)} / \mu^{(n-1)}} \bigr)
(g_\alpha),  \end{equation}
we shall use the following generalization of border-strip tableaux.

\begin{definition}\label{def:n_bs} Let $m$, $n \in \N$.
Let $\lambda / \mu$ be an $n$-decomposable skew-partition 
and let $\alpha = (\alpha_1,\ldots,\alpha_k)$ be a composition of~$m$.
A \emph{$n$-quotient border-strip tableau} of \emph{shape $\lambda / \mu$}
and \emph{type} $\alpha$ is an $n$-tuple $(T_0, \ldots, T_{n-1})$
of border-strip tableaux such that

(a) for each $i \in \{0, \ldots, n-1\}$, the shape of $T_i$ is $\lambda^{(i)} 
/ \mu^{(i)}$, \emph{and}

(b) for each $j \in \{1,\ldots,k\}$, the boxes in the $T_i$ labelled
$j$ lie in a single tableau, where they form a border strip of length $\alpha_j$.
\end{definition}

By the Murnaghan--Nakayama rule we have
\[ \chi^{\lambda / \mu}(g_{n\alpha}) = \sum \sgn(T) \]
where the sum is over all border-strip tableaux of shape
$\lambda / \mu$ and type $n\alpha$. The bijection
in the following proposition implies that
\begin{equation}
\label{eq:midpoint} \chi^{\lambda / \mu}(g_{n\alpha}) =
\nsgn_n(\lambda / \mu) \sum \sgn(T_0) \ldots \sgn(T_{n-1})\end{equation}
where the sum is over all $n$-quotient border-strip tableaux
$(T_0,\ldots,T_{n-1})$ of shape $\lambda / \mu$
and type $\alpha$. 
An illustrative example of the bijection is given in Figure~\ref{fig:bijection} in
\S\ref{subsec:example} below.

\bigskip
\begin{proposition}\label{prop:nbij}
Let $\lambda / \mu$ be an $n$-decomposable skew-partition of $mn$
and let $\alpha = (\alpha_1,\ldots,\alpha_k)$ be a partition of $m$.
There is a canonical bijection between border-strip tableaux
of shape $\lambda / \mu$ and type $n\alpha$ and
$n$-quotient border-strip tableaux of shape $\lambda / \mu$
and type $\alpha$. Under this bijection, if $T$ is mapped to $(T_0, \ldots, T_{n-1})$,
then
\[ \sgn(T) = \nsgn_n(\lambda / \mu)\, \sgn(T_0) \ldots \sgn(T_{n-1}). \]
\end{proposition}

\begin{proof}
Let $T$ be a border-strip tableau of shape $\lambda / \mu$ and type $\alpha$.
The abacus gives a canonical bijection between
border strips 
in $\lambda / \mu$ 
of length $n \ell$ 
and border strips 
of length $\ell$ in the skew-partitions
$\lambda^{(i)} / \mu^{(i)}$ for $i \in \{0,\ldots,n-1\}$.
If the border strip of length $n\alpha_k$ in $T$ corresponds
to a border strip of length~$\alpha_k$ 
in $\lambda^{(i_k)} / \mu^{(i_k)}$, then we label the
corresponding boxes in the Young diagram of~$\lambda^{(i_k)} / \mu^{(i_k)}$
by $k$. Removing these border strips from the tableaux concerned and
iterating the process with the border strip of length $n\alpha_{k-1}$, and so on,
 we obtain a canonical bijection between
border-strip tableaux of shape $\lambda / \mu$ and type $n\alpha$
and $n$-quotient border-strip tableaux of shape $\lambda / \mu$ and type~$\alpha$.

It only remains to prove the assertion about signs.
Since $\nsgn_n(\lambda / \mu)$ is the common sign of any $\lambda/\mu$-tableau
of shape $\lambda/\mu$ and type $(n^m)$, 
it suffices to show that if $T$ is a $\lambda/\mu$-tableau of
type $n\alpha$ and $U$ is a $\lambda/\mu$-tableau of type $n\beta$ then
\[ \sgn(T) \sgn(U) = \sgn(T_0) \ldots \sgn(T_{n-1}) \sgn(U_0) \ldots \sgn(U_{n-1}). \]
Starting from $\Gamma(\lambda)$ with the beads numbered in order of their
positions, perform the sequence of bead moves corresponding to $T$,
then perform the inverse of the sequence of bead moves corresponding to $U$.
Let $\sigma$ be the resulting permutation of the beads.
Each time a border-strip of height $\ell$ is removed or added, the permutation
required to restore the order of numbers is an $\ell+1$-cycle. Therefore
$\sgnop \sigma = \sgn(T) \sgn(U)$.
On the other hand,~$\sigma$ permutes the beads on each runner amongst themselves, and a similar
argument shows that $\sgn(\sigma) =  \sgn(T_0) \ldots \sgn(T_{n-1}) \sgn(U_0) \ldots \sgn(U_{n-1})$.
\end{proof}

Comparing Equations~\eqref{eq:charmain} and~\eqref{eq:midpoint} we see
that 
to complete the proof of Theorem~\ref{thm:ind}, it suffices to show that
\begin{equation}
\label{eq:final} 
\sum \sgn(T_0) \ldots \sgn(T_{n-1}) =
\Ind^{S_m}_H 
\bigl(
\chi^{\lambda^{(0)} / \mu^{(0)}} \times \cdots 
\times \chi^{\lambda^{(n-1)} / \mu^{(n-1)}} \bigr) (g_\alpha) \end{equation}
where the sum is over all $n$-quotient border-strip tableaux
$(T_0,\ldots,T_{n-1})$ of shape $\lambda / \mu$
and type $\alpha$. 
In fact Equation~\eqref{eq:final} follows from Lemma~\ref{lemma:model} and the
Murnaghan--Nakayama rule,
by some manipulations that are essentially formal.
By Lemma~\ref{lemma:model} we have
\begin{equation}\label{eq:formal1} 
\begin{split} \Ind^{S_m}_H 
\bigl( 
\chi^{\lambda^{(0)} /  \mu^{(0)}} {}&{}\times \cdots 
\times \chi^{\lambda^{(n-1)} / \mu^{(n-1)}} \bigr)
(g_\alpha)  \\ &= \sum_{\mathbf{t}} \chi^{\lambda^{(0)} / \mu^{(0)}} (g_{\alpha_0(\mathbf{t})})
\ldots \chi^{\lambda^{(n-1)} / \mu^{(n-1)}} (g_{\alpha_{n-1}(\mathbf{t})} )
\end{split}
\end{equation}
where the sum is over all $(\ell_0,\ldots,\ell_{n-1})$-tabloids $\mathbf{t}$   
such that $g_\alpha \mathbf{t} = \mathbf{t}$ and $\alpha_i(\mathbf{t})$ is 
the cycle type of the permutation of row $i+1$ of $\mathbf{t}$
induced by $g_\alpha$.

For such a tabloid $\mathbf{t}$, let
\[ f(\mathbf{t}) = \sum \sgn(T_0) \ldots \sgn(T_{n-1}) \]
where the sum is over all $n$-quotient border-strip tableaux $(T_0,\ldots, T_{n-1})$
of shape $\lambda / \mu$ and type $\alpha$, 
such that $T_i$ has a border strip labelled $j$ (of length~$\alpha_j)$
if and only if the elements of the orbit of $g_\alpha$ corresponding
to the part $\alpha_j$ lie in row $i+1$ of $\mathbf{t}$.
The Murnaghan--Nakayama rule implies that if $g_\alpha \mathbf{t} = \mathbf{t}$
then
\[ \chi^{\lambda^{(0)} / \mu^{(0)}} (g_{\alpha_0(\mathbf{t})})
\ldots \chi^{\lambda^{(n-1)} / \mu^{(n-1)}} (g_{\alpha_{n-1}(\mathbf{t})} ) =  f(\mathbf{t}), \]
and so, by Equation~\eqref{eq:formal1},
\begin{equation*}\label{eq:formal2} 
\Ind^{S_m}_H 
\bigl( 
\chi^{\lambda^{(0)} /  \mu^{(0)}} \times \cdots 
\times \chi^{\lambda^{(n-1)} / \mu^{(n-1)}} \bigr)
(g_\alpha) = \sum f(\mathbf{t}) 
\end{equation*}
where the sum is over all $(\ell_0, \ldots, \ell_{n-1})$-tabloids $\mathbf{t}$ such that $g_\alpha \mathbf{t} = \mathbf{t}$.
Every $n$-quotient border-strip tableau of shape $\lambda /\mu$
and type $\alpha$ corresponds to some tabloid $\mathbf{t}$
such that  $g_\alpha \mathbf{t} = \mathbf{t}$.
Thus
\begin{equation*}
\label{eq:sums}
 \sum \sgn(T_0) \ldots \sgn(T_{n-1}) = \sum f(\mathbf{t}) \end{equation*}
where the left-hand sum is over all $n$-quotient border-strip tableaux
of shape $\lambda / \mu$ and type $\alpha$, and the right-hand sum is over all
$(\ell_0, \ldots, \ell_{n-1})$-tabloids 		
 $\mathbf{t}$ such that \hbox{$g_\alpha \mathbf{t} = \mathbf{t}$}. 
Equation~\eqref{eq:final} now
follows on comparing the two preceding equations. 

\subsection{Example}
\label{subsec:example}
We give an example of the correspondences used in the proof of Theorem~\ref{thm:ind}.
Let $\lambda / \mu = (8,5,3,2,2,2) / (2,2,1,1,1)$. Any border-strip tableau of shape $\lambda/\mu$
and type $(3^5)$ has either two or four $3$-border-strips of height $1$,
with the rest of height $0$, so  
\hbox{$\nsgn_3(\lambda /\mu) = 1$}. Let $\alpha = (2,1,1,1)$.
 By the Murnaghan--Nakayama rule,~$\chi^{\lambda / \mu}(g_{3 \alpha})$
is the sum of the signs of the four border-strip tableaux of type $\lambda / \mu$
and type~$3 \alpha =(6,3,3,3)$ 
shown in Figure~\ref{fig:bijection} above. 
Their signs are $+1$, $-1$, $-1$, $-1$, respectively,
so  $\chi^{\lambda / \mu}(g_{3 \alpha}) = -2$.
These tableaux are in bijection with the four $3$-quotient border-strip
tableaux of shape $\lambda / \mu$ and type $\alpha=(2,1,1,1)$ shown in Figure~\ref{fig:bijection};
since $\sgn_3(\lambda / \mu) = 1$, the bijection is sign preserving.

\begin{figure}[t]
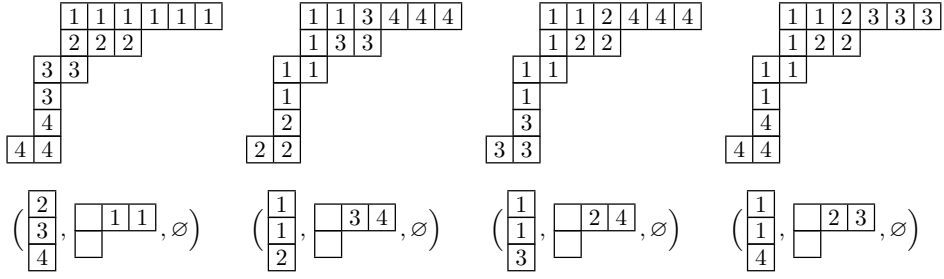
 
\scalebox{0.8}{\hskip-1.1in\parbox{\textwidth}{
\[ \young(::111111,::222,:33,:3,:4,44)\quad
\young(::113444,::133,:11,:1,:2,22)\quad
\young(::112444,::122,:11,:1,:3,33)\quad
\young(::112333,::122,:11,:1,:4,44) 
\]
}}
\scalebox{0.8}{\hskip-1.05in\parbox{\textwidth}{
\[ \Bigl(\, \young(2,3,4)\, ,\, \young(\;11,\;) \, , \emptyset \Bigr)\hskip22pt
\Bigl(\, \young(1,1,2)\, ,\, \young(\;34,\;) \, , \emptyset \Bigr)\hskip22pt
\Bigl(\, \young(1,1,3)\, ,\, \young(\;24,\;) \, , \emptyset \Bigr)\hskip22pt
\Bigl(\, \young(1,1,4)\, ,\, \young(\;23,\;) \, , \emptyset \Bigr)
\]
}}

\caption{
The bijection in Proposition~\ref{prop:nbij} between border-strip tableaux of shape $(8,5,3,2,2,2) / (2,2,1,1,1)$
and type $(6,3,3,3)$ and $3$-quotient border-strip tableaux of the same shape
and type $(2,1,1,1)$. The shapes of the border-strip tableaux forming each $3$-quotient border-strip
tableau are given by the $3$-quotient of $(8,5,3,2,2,2) / (2,2,1,1,1)$, namely   
 $\bigl( (1,1,1), (3,1) / (1,1), \emptyset \bigr)$. To make
 clear the skew-shape, the 
  tableaux of shape $(3,1) / (1,1)$
 are drawn as $(3,1)$-tableaux with two empty boxes.
 }
\label{fig:bijection}
\end{figure}

The $3$-quotient of $\lambda / \mu$ is 
 $\bigl( (1,1,1), (3,1) / (1,1), \emptyset \bigr)$, so 
the characters of $S_3$ and $S_2$ we must consider
are the sign character 
and the trivial character, 
respectively. Taking $g_{(2,1,1,1)} = (12) \in S_5$, and 
following the end of the proof of Theorem~\ref{thm:ind}, we see that there are four
$(3,2)$-tabloids fixed by $(12)$, namely
\[ \yngt(345,12)\, , \quad \yngt(123,45)\, , \quad \yngt(124,35)\, , \quad \yngt(125,34),\]
in the order corresponding to the tableaux shown in Figure~\ref{fig:bijection}.
The corresponding values of the function $f$ used in the proof of Theorem~\ref{thm:ind}
are $+1$, $-1$, $-1$, $-1$ respectively. It should be noted that
in general there will be several tableaux corresponding to each 
product of character values 
\[ \chi^{\lambda^{(0)} / \mu^{(0)}} (g_{\alpha_{0}(\mathbf{t})})\ldots
\chi^{\lambda^{(n-1)} / \mu^{(n-1)}} (g_{\alpha_{n-1}(\mathbf{t})}); \]
it is a special feature of this example that each $f(\mathbf{t})$ has
a single summand, and so the bijection extends all the way to
tabloids.

\section{Proof of Theorem~\ref{thm:cd}}
\label{sec:proof}

We shall prove Theorem~\ref{thm:cd}  by induction on the number
of parts of $\gamma$. 
Most of the work occurs  in proving
the base case when $\gamma$ has a single part. In the first step
we combine the results of \S 3 and \S 4. For later use in \S\ref{sec:gendef},
we state the following proposition for a general deflation map.

\begin{proposition}\label{prop:gen_basecase}
Let $m$, $n \in \N$, let $\theta$ be an irreducible
character of $S_m$, and let $\lambda / \mu$ be a skew-partition of $mn$.
Let $g\in S_n$ be an $n$-cycle.
If $\lambda / \mu$ is not $n$-decomposable, then
$\Defres^\theta_{S_n}(g) = 0$. 
If $\lambda / \mu$ is $n$-decomposable, 
then
\[ \begin{split}
(\Defres^\theta_{S_n} 
\chi^{\lambda / \mu} )(g) 
= \nsgn_n(\lambda / \mu)
\left< \Ind^{S_m}_H 
 \bigl(
\chi^{\lambda^{(0)} / \mu^{(0)}} \times \cdots \times 
\chi^{\lambda^{(n-1)}/ \mu^{(n-1)}} \bigr), \theta \right>,
\end{split}
 \]
where $H = S_{|\lambda^{(0)} / \mu^{(0)}|} \times \cdots
\times S_{|\lambda^{(n-1)} / \mu^{(n-1)}|}$.
\end{proposition}

\begin{proof}
If $\lambda / \mu$ is not $n$-decomposable
then, by Theorem~\ref{thm:ind}, $\chi^{\lambda / \mu}(g_{n\alpha}) = 0$  
for all partitions $\alpha$ of $m$. Hence, by Proposition~\ref{prop:omega}(ii), 
\hbox{$(\Defres^\theta_{S_n}\chi^{\lambda/\mu})(g) = 0$}. 
If $\lambda / \mu$ is $n$-decomposable then,
using Proposition~\ref{prop:omega}(i), we may restate Theorem~\ref{thm:ind} as
\begin{equation}
\label{eq:omegares}
\omega(\Res_{S_m \wr S_n} \chi^{\lambda / \mu}) = \varepsilon_n(\lambda / \mu)
\Ind^{S_m}_H \bigl(
\chi^{\lambda^{(0)} / \mu^{(0)}} \times \cdots \times 
\chi^{\lambda^{(n-1)}/ \mu^{(n-1)}} \bigr). 				
\end{equation}
The result now follows from
 Proposition~\ref{prop:omega}(ii).
\end{proof}

It is clear from the condition on row numbers in Equation~\eqref{eq:row_numbers} in Definition~\ref{def:a}
that if $\lambda / \mu$ is a skew-partition of $mn$ then there is at most
one $m$-border-strip tableau of shape $\lambda / \mu$ and type $(n)$.    
The following lemma gives a more precise condition.
Recall that a skew-partition $\sigma / \tau$ is said to be a \emph{horizontal
strip} if the Young diagram of $\sigma / \tau$ has no two boxes in the same column.

\begin{lemma}\label{lemma:horiz}
Let $m$, $n \in \N$ and let $\lambda / \mu$ be a 
skew-partition of $mn$.
If $\lambda / \mu$ is $n$-decomposable, and each $\lambda^{(i)} / \mu^{(i)}$ 
is a horizontal strip, then there is a unique $m$-border-strip tableau
of shape $\lambda / \mu$ and type $(n)$; this tableau has sign $\nsgn_n(\lambda / \mu)$.
Otherwise there are no such tableaux.
\end{lemma}

\begin{proof}
Suppose that $T$ is an $m$-border-strip tableau of type $(n)$ and shape $\lambda / \mu$. 
Then by Definition~\ref{def:n_dec}, $\lambda / \mu$ is $n$-decomposable. 
Let $\Gamma(\lambda)$ be an $n$-runner abacus display for $\lambda$ using $tn$ beads
for some $t \in \N$. Let $\Gamma(\mu)$ be the abacus display for $\mu$ obtained
from $\Gamma(\lambda)$ by an appropriate sequence of $m$ single bead moves, so that 
in each move a bead is slid upwards into a gap immediately above it.
We label the positions on the abacus from top to bottom so that the
positions in row $r$ of an abacus display are numbered $(r-1)n, \ldots, rn-1$ (as usual).
Observe that the row number of a border strip of length $n$ in~$T$ corresponding to a bead in position $p$ of $\Gamma(\lambda)$
is the number of beads in positions $p+1$, $p+2$, \ldots \ of $\Gamma(\lambda)$.   
Therefore if $p < p'$ and the beads in positions $p$ and $p'$
of $\Gamma(\lambda)$ both correspond to border strips in $T$, then
in the sequence of bead moves corresponding to $T$, 
the bead in position $p'$ is moved upwards before the bead in position~$p$.

Let $i \in \{0,\ldots,n-1\}$. Suppose that $\Gamma(\lambda)$ has beads in positions $nq + i$ and $nq' + i$ where $q < q'$,
and that $\Gamma(\mu)$ has no beads in positions $n(q+1) + i$, \ldots, $nq' + i$. 
The bead in position $nq+i$ of $\Gamma(\lambda)$ prevents the
bead initially in position $nq'+i$ from reaching its final position in $\Gamma(\mu)$.
Therefore the bead in position $nq+i$ must be moved before the bead in position $nq'+i$ reaches its final position.
This contradicts the previous paragraph.
 Hence there exist $x_j \in \N_0$ and $y_j \in \N_0$ such
that the beads on runner $i$ of $\Gamma(\lambda)$ are in positions
$\{i + nx_j : 1\le j \le s \}$, the beads on runner $i$ of $\Gamma(\mu)$
are in positions $\{ i + ny_j : 1 \le j \le s \}$ and
\begin{equation}
\label{eq:beads}
 y_1 \le x_1 < y_2 \le x_2 < \cdots < y_s \le x_s.
\end{equation}
It easily follows that $\lambda^{(i)} / \mu^{(i)}$ is a horizontal strip.
Then, by Proposition~\ref{prop:nbij}, $\sgn(T) = \nsgn_n(\lambda / \mu)$.
 
Conversely, suppose that $\lambda / \mu$ is $n$-decomposable
and each $\lambda^{(i)} / \mu^{(i)}$ is a horizontal strip.
Then the inequality~\eqref{eq:beads}
on the bead positions in each runner holds.
We now describe a sequence of $m$ single upward bead moves that transforms $\Gamma(\lambda)$ into $\Gamma(\mu)$, and thus corresponds to a border-strip tableau $T$ of shape $\lambda / \mu$ and type $(n)^{\star m}$. At each step, locate the bead with maximal position $p$ such that there is no bead in position $p$ of $\Gamma(\mu)$. Slide that bead up into position $p-n$. This is possible by inequality~\eqref{eq:beads}. The row numbers of  the 
border strips of length $n$ corresponding to this sequence of  moves are
increasing, so $T$ is a $m$-border-strip tableau of shape $\lambda / \mu$ and type~$(n)$. 
Uniqueness is clear since there is always at most one $m$-border strip tableau of shape $\lambda / \mu$ and type $(n)$.
\end{proof}

We can now complete the proof of the base case.

\begin{proposition}\label{prop:basecase}
Let $m$, $n \in \N$ and let $g\in S_n$ be an $n$-cycle.
If $\lambda / \mu$ is a
skew-partition of $mn$ then 
\[ (\Defres_{S_n} \chi^{\lambda / \mu})(g) 
= a_{\lambda / \mu, (n)} .\]
\end{proposition}

\begin{proof}
If $\lambda / \mu$ is not $n$-decomposable then 
\hbox{$a_{\lambda / \mu,(n)} = 0$}
by Lemma~\ref{lemma:horiz}.
 Proposition~\ref{prop:gen_basecase} implies
 that 
the result holds in this case.

Now suppose that $\lambda / \mu$ is $n$-decomposable.
Let $\ell_i = |\lambda^{(i)} / \mu^{(i)}|$ for each $i \in \{0,\ldots,n-1\}$.
By Proposition~\ref{prop:gen_basecase} 
 and Frobenius reciprocity, we have
\[ (\Defres_{S_n} \chi^{\lambda/\mu})(g)=
\nsgn_n(\lambda/\mu)
\left< \chi^{\lambda^{(0)} / \mu^{(0)}}, 1_{S_{\ell_0}} \right> 
\ldots \left< \chi^{\lambda^{(n-1)} / \mu^{(n-1)}}, 1_{S_{\ell_{n-1}}} \right>. \]

It follows from \cite[Theorem~2.3.13(ii)] {JK}
that if $\sigma / \tau$ is a skew-partition of $\ell$ then
\[ \left< \chi^{\sigma / \tau}, 1_{S_\ell} \right> =  \begin{cases}
1 & \text{if $\sigma / \tau$ is a horizontal strip} \\
0 & \text{otherwise}.\end{cases}
\]
Therefore $(\Defres_{S_n}\chi^{\lambda/\mu})(g) = \nsgn_n(\lambda / \mu)$ if each
$\lambda^{(i)} / \mu^{(i)}$ is a horizontal strip, and otherwise 
$(\Defres_{S_n}\chi^{\lambda/\mu})(g) = 0$. The proposition
now follows from Lemma~\ref{lemma:horiz}.
\end{proof}

For the inductive step we need the following lemma and proposition.
The former is well-known and can be deduced from~\cite[2.3.12]{JK}.  
We write $\mu\subseteq \lambda$ if $\mu$ is a subpartition of $\lambda$ (i.e.\ the Young diagram of $\mu$ is contained in that of $\lambda$). 

\begin{lemma}\label{lemma:skewres}
Let $\lambda / \mu$ be a skew-partition of $r$. If $1 \le c < r$ then
\[ \Res_{S_{c} \times S_{r-c}} \chi^{\lambda / \mu} =
\sum_\tau \chi^{\tau / \mu} \times \chi^{\lambda / \tau} \]
where the sum is over all partitions $\tau$ such that
 $\mu \subseteq \tau \subseteq \lambda$
and $|\tau / \mu| = c$.
\end{lemma}

For later use we state and prove the following 
proposition for a general deflation map.    

\begin{proposition}\label{prop:defres_decomp} 
 Let $m$, $n \in \N$ and let $\lambda / \mu$
be a skew-partition of $mn$. Let $\theta$ be an irreducible
character of $S_m$. Let $g\in S_n$.
If $g = kh$ where $k \in S_{\ell}$ 
and $h \in S_{n-\ell}$,  then 
\begin{equation*}
(\Defres^\theta_{S_n} \chi^{\lambda/\mu})(g) = \sum_{\tau} 
(\Defres^\theta_{S_\ell} \chi^{\tau / \mu})(k)
(\Defres^\theta_{S_{n-\ell}} \chi^{\lambda / \tau})(h)  
\end{equation*}
where the sum is over all partitions $\tau$ such
that $\mu \subseteq \tau \subseteq \lambda$ and $|\tau / \mu|  = m\ell$.
\end{proposition}

\begin{proof}
Let $B$ be the base group of the wreath product $S_m \wr S_n \le S_{mn}$.
Choose a subgroup $S_{m\ell} \times S_{m(n-\ell)} \le S_{mn}$ containing $B$.
If $\psi$ is a character of $S_{m} \wr S_n$ then it is easily checked that
\[ \Res_{S_\ell \times S_{n-\ell}} (\Def_{S_n}^{\theta} 
\psi) = (\Def_{S_\ell}^{\theta} \times 
\Def_{S_{n-\ell}}^{\theta})   
(\Res^{S_m \wr S_n}_{S_m \wr S_\ell \times S_m \wr S_{n-\ell}} \psi).\]
Hence
\[ \Res_{S_\ell \times S_{n-\ell}} 
(\Defres^\theta_{S_{n}} \chi) = 
(\Defres^\theta_{S_\ell} 
\times \Defres^\theta_{S_{n-\ell}}
)  
(\mathrm{Res}^{S_{mn}}_{S_{m\ell} \times S_{m(n-\ell)}} \chi) 
\]
for any character $\chi$ of $S_{mn}$. The proposition now
follows from the 
expression for $\Res_{S_{m\ell} \times S_{m(n-\ell)}} \chi^{\lambda / \mu}$
given in Lemma~\ref{lemma:skewres}.
\end{proof}

We are now ready to prove Theorem~\ref{thm:cd}. Let $m$, $n\in \N$
and let $\lambda / \mu$ be a skew-partition of $mn$. Let $\gamma = (\gamma_1,\ldots,
\gamma_d)$ be a composition of $n$. Let $g \in S_n$ have cycle type $\gamma$
and let $h \in S_{n-\gamma_1}$ have cycle type $(\gamma_2,\ldots,\gamma_d)$.
 Note that, by Lemma~\ref{lemma:horiz},
if $\tau / \mu$ is a skew-partition 
of $m \gamma_1$
then there is at most one $m$-border-strip tableau of shape $\tau / \mu$
and type $(\gamma_1)$.
 We shall denote this tableau by $T_{\tau / \mu}$ when it exists. By Definition~\ref{def:a},
$a_{\tau / \mu, (\gamma_1)} = \sgn(T_{\tau / \mu})$ 
(or is zero if no such tableau exists).
Therefore Proposition~\ref{prop:basecase} may be restated as
 $(\Defres_{S_{\gamma_1}} \chi^{\tau/\mu})(k)=
\sgn(T_{\tau / \mu})$, where $k \in S_{\gamma_1}$ is a $\gamma_1$-cycle   
and $(\Defres_{S_{\gamma_1}} \chi^{\tau/\mu})(k)=0$ if no tableau $T_{\tau / \mu}$ exists.
It therefore follows from
Proposition~\ref{prop:defres_decomp} that
\begin{equation}
\label{eq:defres_decomp_sgn}
 (\Defres_{S_n} \chi^{\lambda / \mu})(g) = \sum_\tau \sgn(T_{\tau/ \mu}) 
 (\Defres_{S_{n-\gamma_1}} \chi^{\lambda / \tau})(h) 
 \end{equation}
where the sum is over all partitions $\tau$ such that $\mu \subseteq \tau \subseteq \lambda$,
$|\tau / \mu| = m\gamma_1$ and
 there is an $m$-border-strip tableau of shape $\tau / \mu$. By induction
 on the number of parts of $\gamma$ we have 
\[
 (\Defres_{S_n} \chi^{\lambda / \mu})(g) = \sum_\tau \sgn(T_{\tau/ \mu}) 
 a_{\lambda / \tau, (\gamma_2,\ldots,\gamma_d)}
\]
with the same conditions on the sum. It is clear that if $T$
is an $m$-border-strip tableau of shape $\lambda / \mu$ and type $\gamma$
then the border strips in $T$ corresponding to the $m$ parts of length $\gamma_1$
in $\gamma^{\star m}$ form an $m$-border-strip tableau of shape $\tau / \mu$
for some $\tau$. Therefore the right-hand side of the previous equation
is~$a_{\lambda / \mu, \gamma}$. This completes the proof of Theorem~\ref{thm:cd}.

\section{An application to Foulkes' Conjecture}
\label{sec:Foulkes}

For $m$, $n \in \N$, let $\phi^{(m^n)}$
be the permutation character of $S_{mn}$
acting on all unordered set partitions of $\{1,2,\ldots,mn\}$ into
$n$ sets each of size $m$. Equivalently, 
$\phi^{(m^n)} = \Ind_{S_m \wr S_n}^{S_{mn}} 1$.
Foulkes' Conjecture asserts that if $m \le n$ then
\[ \left< \phi^{(m^n)}, \chi^\lambda \right> \ge 
\left< \phi^{(n^m)}, \chi^\lambda \right> \]
for all partitions $\lambda$ of $mn$.
Equivalent formulations of Foulkes' Conjecture exist in the language of  general linear groups, symmetric polynomials, and geometric invariant theory. 
Despite having been 
attacked from all these directions (and more), it has
only been proved when $m \le 4$ (see \cite{DentSiemons} and \cite{McKay}),  asymptotically
when $n$ is very large compared to $m$ (see \cite[page 352]{Brion}) and, in a computational result of M\"uller and Neunh\"offer \cite{MN},  when $m+n \le 17$.  For further background, and some recent results on the constituents
of $\phi^{(m^n)}$, we refer the reader to \cite{PagetWildon} and \cite{Giannelli}.

In this section we use character deflations to prove a new recursive formula for the
character multiplicities in Foulkes' Conjecture.
Firstly, using Frobenius reciprocity, then the inflation-deflation reciprocity relation 
\begin{equation} \label{eq:reciprocity}
\left< \Def_{S_n} \psi, \chi  \right> = 
\left< \psi, \Inf_{S_n}^{S_m \wr S_n} \!\chi  \right>,
\end{equation}
where $\psi$ is a character of $S_m \wr S_n$ and $\chi$ is a character of $S_n$, we observe that
\[
\left< \phi^{(m^n)}, \chi^\lambda \right> = 
\left< \Defres_{S_n} \chi^{\lambda} ,  1_{S_n} \right>.
\]

\begin{proposition}\label{prop:Foulkes}
Let $m$, $n \in \mathbf{N}$. If $\lambda$ is a 
partition of $mn$ then
\[ 
\left< \phi^{(m^n)}, \chi^\lambda \right>
 = \displaystyle\frac1{n} \displaystyle\sum_{\ell=1}^n \displaystyle\sum_{\mu}
\epsilon_\ell(\lambda / \mu) 
\left< \phi^{(m^{n-\ell})}, \chi^\mu \right>
\]
where the second sum is over all partitions $\mu$ of $\ell m$ such that there exists an $m$-border strip tableau of shape $\lambda / \mu$ and type $(\ell)$. 
\end{proposition}

\begin{proof}
We have seen that
\[
\left< \phi^{(m^n)}, \chi^\lambda \right> = 
\left< \Defres_{S_n} \chi^{\lambda} ,  1_{S_n} \right>=\displaystyle\frac1{n!} \displaystyle\sum_{g \in S_n} (\Defres_{S_n} \chi^{\lambda}) (g).
\]
We may write each $g \in S_n$ as a product of an $\ell$-cycle containing the letter~$1$ and some $h \in S_{n - \ell}$ acting on the remaining letters. 
The number of possible such  $\ell$-cycles is $(n-1)!/ (n-\ell)!$, 
hence
\[
\left< \phi^{(m^n)}, \chi^\lambda \right> = \displaystyle\frac1{n!}  \displaystyle\sum_{\ell =1}^n \displaystyle\frac{(n-1)!}{(n-\ell)!} \displaystyle\sum_{h \in S_{n-\ell}} (\Defres_{S_n} \chi^{\lambda}) (xh)
\]
where $x$ is the $\ell$-cycle $(1 \, 2 \, \ldots \, \ell)$.
We now apply Proposition~\ref{prop:defres_decomp} to see that
\[
\left< \phi^{(m^n)}, \chi^\lambda \right> =
 \displaystyle\frac1{n}  \displaystyle\sum_{\ell =1}^n \displaystyle\frac1{(n-\ell)!} \displaystyle\sum_{h \in S_{n-\ell}} 
\displaystyle\sum_{\mu} (\Defres_{S_\ell} \chi^{\lambda/\mu}) (x)
 (\Defres_{S_{n-\ell}} \chi^{\mu}) (h),
\]
where the sum is over partitions $\mu \subseteq \lambda$ with $|\lambda /\mu|=m \ell$.  
Since $x$ is an $\ell$-cycle, Proposition~\ref{prop:basecase} shows that $ (\Defres_{S_\ell} \chi^{\lambda/\mu}) (x)$ is
$\epsilon_\ell(\lambda/\mu)$ if  there exists an $m$-border strip tableau of shape $\lambda / \mu$ and type $(\ell)$ and is zero otherwise. Thus
\begin{eqnarray*}
\left< \phi^{(m^n)}, \chi^\lambda \right>& =&
 \displaystyle\frac1{n}  \displaystyle\sum_{\ell =1}^n 
\displaystyle\sum_{\mu}   \epsilon_\ell(\lambda/\mu)
\displaystyle\frac1{(n-\ell)!} \displaystyle\sum_{h \in S_{n-\ell}} 
 (\Defres_{S_{n-\ell}} \chi^{\mu}) (h) \\  & =& 
 \displaystyle\frac1{n}  \displaystyle\sum_{\ell =1}^n 
\displaystyle\sum_{\mu}   \epsilon_\ell(\lambda/\mu)
 \left< \Defres_{S_{n-\ell}} \chi^{\mu} ,  1_{S_{n-\ell}} \right>\\
&=&  \displaystyle\frac1{n} \displaystyle\sum_{\ell=1}^n \displaystyle\sum_{\mu}
\epsilon_\ell(\lambda / \mu) 
\left< \phi^{(m^{n-\ell})}, \chi^\mu \right>
\end{eqnarray*}
where, in each case, the second sum is over all partitions $\mu \subseteq \lambda $ for which there exists an $m$-border strip tableau of shape $\lambda / \mu$ and type $(\ell)$. 
\end{proof}

Proposition~\ref{prop:Foulkes} gives an algorithm for testing Foulkes' Conjecture for a single character $\chi^{\lambda}$ of $S_{mn}$ that is far faster
than more direct methods,
such as those requiring the character values of $\phi^{(m^n)}$ and $\phi^{(n^m)}$ to be calculated on all partitions of~$mn$. Timings suggest that it can be significantly faster
than the algorithm used by {\sc symmetrica} \cite{Kerber}, although some of this
gain comes at the expense of increased use of memory. For example,
to calculate all the multiplicities $\langle \phi^{(6^{11})}, \chi^\lambda \rangle$ for
$\lambda$ a partition of $66$ takes $34$ minutes using the Haskell \cite{PeytonJones}
implementation of Proposition~\ref{prop:Foulkes} available from the third
author's website\footnote{See \url{www.ma.rhul.ac.uk/~uvah099/}.}, compared to $350$ minutes for {\sc symmetrica} using the function
\verb+COMPLETE_COMPLETE_PLET+, both running on the same machine.

The graphs in Figures~\ref{fig:FoulkesGraph} and~\ref{fig:FoulkesGraphsComp}
show a number of intriguing features of the character multiplicities
appearing in Foulkes' Conjecture.
In particular, it seems plausible that if Foulkes' Conjecture is false, 
then a counterexample occurs when $m$ is near to $n$ and
the relevant partition is either very large or very small
in the lexicographic order on partitions of $mn$
with at most~$n$~parts.

\begin{figure}
\bigskip 
\medskip
\makebox[\textwidth]{
\hskip0.5cm\scalebox{0.75}{\includegraphics{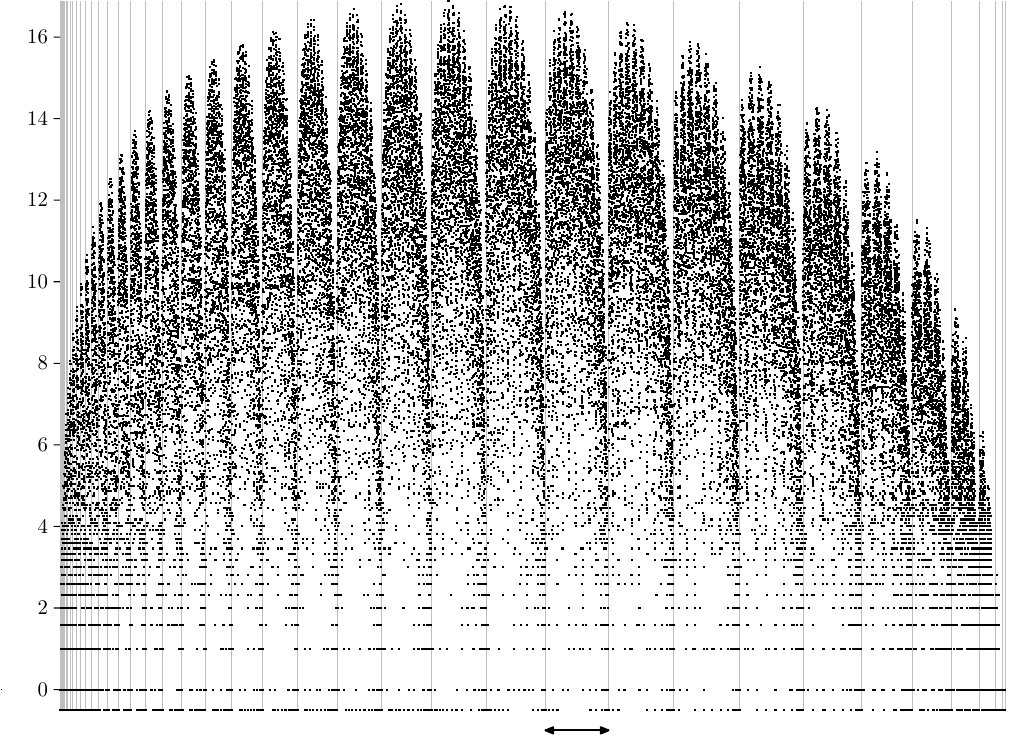}
\raisebox{1pt}{\includegraphics{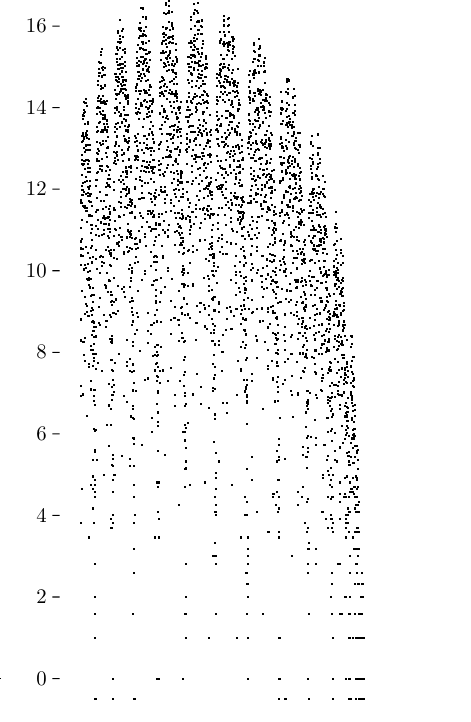}}}} 
\caption{The left-hand graph shows $\log_2 \left< \phi^{(7^8)}, \chi^\lambda \right>$ for all
partitions of $56$ with at most $8$ parts. Partitions are ordered
lexicographically, with the smallest partition $(7^8)$ at the far
right. If the multiplicity is
zero then the point is  placed below the $x$ axis. Vertical lines
separate partitions with equal largest parts. The right-hand graph shows 
an enlarged view of the multiplicities for partitions with first part $19$;
the range of these partitions is indicated by the arrow in the left-hand graph.}
\label{fig:FoulkesGraph}
\end{figure}

\begin{figure}
\setlength{\unitlength}{1.25cm}
\begin{picture}(10,16)
\put(-0.75,8){\scalebox{0.8}{\includegraphics{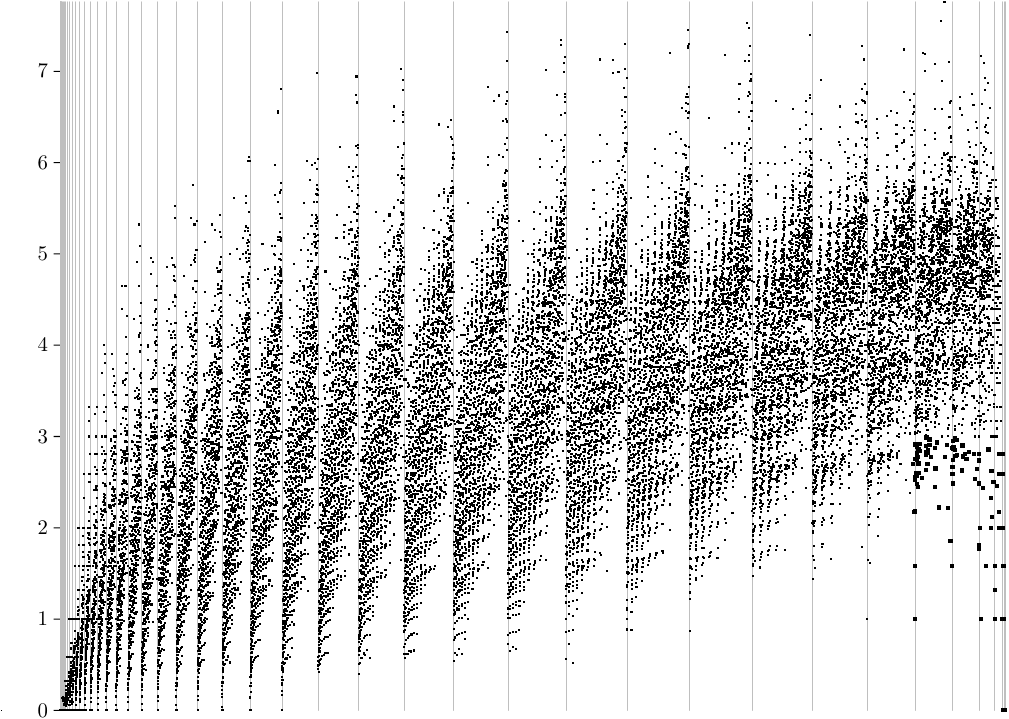}}}
\put(-0.75,.2){\scalebox{0.8}{\includegraphics{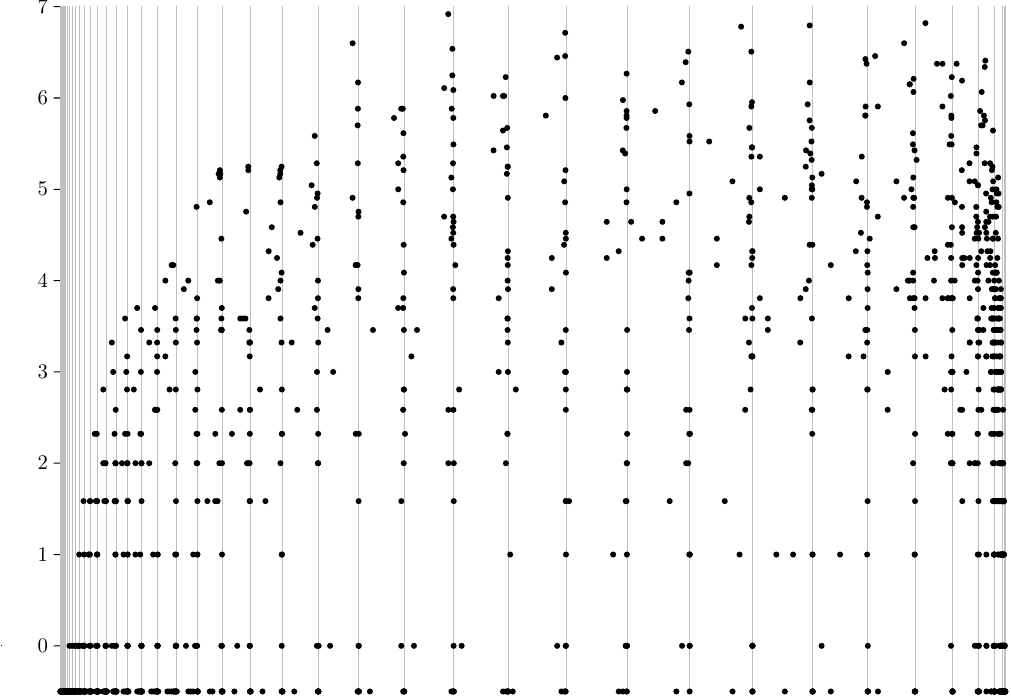}}}
\end{picture}
\caption{The top graph shows 
$\log_2 \bigl< \phi^{(7^8)}, \chi^\lambda \bigr> - \log_2 \bigl< \phi^{(8^7)}, \chi^\lambda\bigr>$ for 
the partitions of $56$ with at most $7$ parts for which the smaller multiplicity is non-zero.
To increase their visibility a small number of points have been enlarged.
Partitions are ordered lexicographically, with the smallest partition $(8^7)$ at the far
right. The lower graph shows $\log_2 \bigl< \phi^{(7^8)}, \chi^\lambda \bigr>$ for
those partitions for which $\bigl< \phi^{(8^7)}, \chi^\lambda\bigr> = 0$; if 
$\bigl< \phi^{(7^8)}, \chi^\lambda \bigr> = 0$
then the point is drawn below the axis.}
\label{fig:FoulkesGraphsComp}
\end{figure}

Using {\sc symmetrica}, Foulkes' Conjecture has been checked for
all $m$ and~$n$ with $m+n \le 17$ in~\cite{MN}.
Using Proposition~\ref{prop:Foulkes} and the software already
mentioned, we have extended this range. The relevant data
is available from the third author's 
website. 

\begin{corollary}\label{cor:computation}
If $m \le n$ and $m+n \le 19$ 
then
\[ \left< \phi^{(m^n)}, \chi^\lambda \right> \ge 
\left< \phi^{(n^m)}, \chi^\lambda \right> \]
for all partitions $\lambda$ of $mn$.\hfill\qed
\end{corollary}

\bigskip
\section{Generalized deflations}\label{sec:gendef}

In this section we discuss deflation with respect to an arbitrary irreducible character $\theta$ of $S_m$. In the case where $\theta$ is labelled by a hook partition, $\theta=\chi^{(a,1^b)}$ where $a+b=m$,  we give a combinatorial description, generalizing Theorem~\ref{thm:cd}.
We also prove some other general results, and show how these, 
together with the results of \S\ref{sec:proof}, may be used to calculate the values of an irreducible character of~$S_{mn}$ deflated with respect to an arbitrary character of $S_m$.

Firstly we deduce from Theorem~\ref{thm:cd} an analogous
result for deflations with respect to the sign character. 
As is usual, if $\lambda$ is a partition then we denote by $\lambda'$
the conjugate partition to $\lambda$. 

\begin{proposition}\label{prop:sgn}
Let $m$, $n \in \N$ and let $\lambda / \mu$ be a skew-partition of $mn$. If~$\gamma$ 
is a composition of $n$ and $g \in S_n$ has cycle type $\gamma$ then
\[ (\Defres_{S_n}^{\sgn_{S_m}})(g) = 
\begin{cases} a_{\lambda' / \mu', \gamma} & \text{if $m$ is even} \\
\sgn_{S_n}(g) \, a_{\lambda' / \mu', \gamma} & \text{if $m$ is odd}. 
\end{cases}  \]
\end{proposition}

\begin{proof}
It is easily seen that
\[ \Res_{S_m \wr S_n} \sgn_{S_{mn}} = \begin{cases}
\widetilde{\sgn_{S_m}^{\times n}} & \text{if $m$ is even} \\[6pt]
\widetilde{\sgn_{S_m}^{\times n}}\, \sgn_{S_n} & \text{if $m$ is odd.}
\end{cases} \]
Hence if $\chi$
is any character of $S_{mn}$ then
\[ \Defres^{\sgn_{S_m}}_{S_n} \chi = \eta \Defres_{S_n} (\chi\, \sgn_{S_{mn}}) \]
where $\eta = 1_{S_n}$ if $m$ is even and $\eta = \sgn_{S_n}$ if $m$ is odd.
It follows from Theorems~7.15.6
 and~7.17.5 of \cite{StanleyII} that			
if $\lambda / \mu$ is a skew-partition of $mn$ then
$\chi^{\lambda / \mu}\, \sgn_{S_{mn}} = \chi^{\lambda' / \mu'}$. Therefore,
by Theorem~\ref{thm:cd},
we have 
\[ \Defres^{\sgn_{S_m}}_{S_n}(\chi^{\lambda'/\mu'}) = 
\eta(g) a_{\lambda' / \mu',\gamma},\] 
as required. 
\end{proof}

It would also have been possible to
prove Proposition~\ref{prop:sgn} directly
from Proposition~\ref{prop:gen_basecase}, by reasoning along the same lines as
the proof of Theorem~\ref{thm:cd} in \S 5.

We now turn our attention to the case where $\theta=\chi^{(a,1^b)}$ for a hook partition $ (a,1^b)$. For convenience, we use the language of skew shapes (see, for example, \cite{Zelevinsky1981}). 
A {\it skew shape} is a finite subset $\kappa$  of $\mathbb{N}\times\mathbb{N}$ that is convex with respect to the partial order $\le_p$ defined by $ (i,j) \le_p (i',j')$  if and only if  $ i\le i'$ and $ j\le j'$.
We identify the 
 skew-partition $\lambda / \mu$ in which $\lambda$ has $t$ parts 
 with the skew shape
$\{ (i,j) : 1 \le j \le t, \mu_i+1 \le j \le \lambda_i\}$. 
We define $\chi^\kappa = \chi^{\lambda / \mu}$ and
$\nsgn_n(\kappa) = \nsgn_n(\lambda / \mu)$.

Suppose $\kappa$ is a skew shape. Then we define the {\it initial box} of $\kappa$ to be $(i_\kappa, j_\kappa)$ where  $(i_\kappa, j_\kappa) \in \kappa$ and, for any $i \le i_\kappa$ and $j \ge j_\kappa$, if $(i,j) \in \kappa$ then  $i = i_\kappa$ and $j = j_\kappa$.
Similarly, we define the  {\it terminal box} of $\kappa$ to be $(k_\kappa, \ell_\kappa)$ where  $(k_\kappa, \ell_\kappa) \in \kappa$ and, 
for any $k \ge k_\kappa$ and $\ell \le \ell_\kappa$, if $(k,\ell) \in \kappa$ then  $k = k_\kappa$ and $\ell = \ell_\kappa$. 
The initial and terminal boxes exist due to the convexity of $\kappa$.  
For all $(i,j) \in \kappa$, $i_\kappa \le i \le k_\kappa$ and $\ell_\kappa \le j \le j_\kappa$.

We use the term {\it border strip} in this context to mean a connected skew shape which contains no $2 \times 2$ square.
We say that $D$ is a {\it  border-strip $n$-diagram}
if $D$ is a finite set of disjoint border strips, each of length $n$,
such that $\bigcup D$ is a skew shape. We say that $D$ is a {\it  horizontal border-strip $n$-diagram} if whenever $(i_\rho,j_\rho)$ is the initial box of some border strip $\rho \in D$, we have $(i,j_\rho) \notin \bigcup D$ for all $i<i_\rho$. Similarly,  $D$ is a {\it  vertical border-strip $n$-diagram} if whenever $(k_\rho,\ell_\rho)$ is the terminal box of some border strip $\rho \in D$, we have $(k_\rho, \ell) \notin \bigcup D$ for all $\ell<\ell_\rho$.

We define a
relation $\mathcal{R}$ on the set of border strips of length~$n$
by $(\rho_1, \rho_2) \in \mathcal{R}$ if $\rho_1$ and $\rho_2$ are disjoint border strips of length $n$ and there exist $z \in \rho_1$ and $w \in \rho_2$ such that $z<_p w$. Observe that 
when $n\ge 3$ the relation $\mathcal{R}$ is not transitive.
We can now state 
the following combinatorial definition, 
which is illustrated in Figure~\ref{fig:skewDiag} overleaf.

\begin{definition}\label{def:hook-like}
Let $\kappa$ be a skew shape of size~$mn$,  $a,b$ be positive integers such that $a+b=m$, and $D$ and $E$ be two border-strip $n$-diagrams such that $\kappa=\bigcup (D\cup E)$. We say that 
$(D,E)$ is an \emph{$(a,1^b)$-like border-strip $n$-diagram of shape $\kappa$} if the following conditions are satisfied:
\begin{enumerate} 
 \item $|D| = a$ and $|E| = b+1$;
 \item $D\cap E = \{ \sigma \}$ where $\sigma$ is border strip of length $n$ which contains the initial box of $\kappa$;
 \item $D$ is a horizontal border-strip $n$-diagram, and $E$ is a vertical border-strip $n$-diagram;
 \item there do not exist $\rho_D\in D$ and $\rho_E \in E$ such that $(\rho_E, \rho_D) \in \mathcal{R}$. 
\end{enumerate}
We denote the set of  $(a,1^b)$-like border-strip $n$-diagrams of shape $\kappa$ by $\mathcal{B}_{a,b}^\kappa$.
\end{definition}

\begin{figure}[t]
\includegraphics{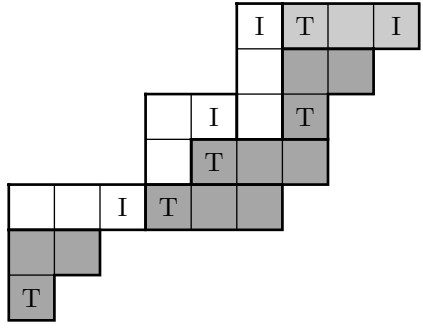}
\caption{A $(4,1^4)$-like border-strip $3$-diagram of shape $(9,8,7,7,6,2,1) /     
(5,5,3,3)$. Ribbons in the horizontal border-strip $3$-diagram are shown in light grey
and white, with their initial boxes labelled $I$; 
ribbons in the vertical border-strip $3$-diagram are shown in light grey and
dark grey with their terminal boxes labelled $T$.}
\label{fig:skewDiag}
\end{figure}

We note that, by a simple counting argument, whenever $(D,E)\in \mathcal{B}_{a,b}^{\kappa}$, 
the elements of $D$ are pairwise disjoint, as well as those of $E$, and we have $(\bigcup D) \cap (\bigcup E) = \sigma$.
Definition~\ref{def:hook-like} is inspired by the Littlewood--Richardson Rule 
 (see \cite[2.8.13]{JK} or \cite[A1.3.3]{StanleyII}) 
 in the case of a hook partition $(a,1^b)$. In this case the rule states
that the multiplicity of $\chi^{(a,1^b)}$ in $\chi^\kappa$
is the number of ways to represent $\kappa$ as a union of two skew shapes 
$\alpha$ and $\beta$ satisfying the following: 
(1) the size of $\alpha$ is $a$ and the size of $\beta$ is $b+1$; (2) $\alpha\cap \beta= \{ x \}$ where $x$ is the initial box of $\kappa$; (3) $\alpha$ is a horizontal strip and $\beta$ is a vertical strip; (4) there do not exist $y\in \alpha$ and $z\in \beta$ such that $z<_{p} y$.
Thus 
\[ (\Def_{S_1}^{\chi^{(a,1^b)}}\!\chi^\kappa)(1_{S_1}) = | \mathcal{B}^\kappa_{a,b}|. \]
This is the case $n=1$ of the following theorem. 	

\begin{theorem}\label{thm:hookdefl}
Let $\kappa$ be a skew shape of size~$mn$. Let $a,b$ be positive integers such that $a+b=m$.  Let $g\in S_n$ be an $n$-cycle.
Then 
\[ (\Defres^{\chi^{(a,1^b)}}_{S_n} \!\chi^{\kappa})(g) = \epsilon_n(\kappa) |\mathcal{B}_{a,b}^\kappa|. \]
\end{theorem}

We observe that an $(m)$-like border-strip $n$-diagram of shape $\kappa$ may be viewed as an $m$-border-strip tableau of shape $\kappa$ and type $(n)$ by labelling the border strips 
in the unique way so that the condition on row numbers in Definition~\ref{def:a} holds.
Thus Theorem~\ref{thm:hookdefl} is 
a generalization of Theorem~\ref{thm:cd} in the case of an $n$-cycle.

\begin{example}\label{example:hook}
We compute $(\Defres^{\chi^{(3,1)}}_{S_3} \chi^{(4,4,4)})(g)$ where $g$ is a 3-cycle. There are two $(3,1)$-like border-strip $3$-diagrams of shape $(4,4,4,4)$, namely $( \{\rho_1, \rho_2, \rho_3\}, \{\rho_3, \rho_4\})$ and
 $( \{\rho_5, \rho_6, \rho_7\}, \{\rho_7, \rho_8\})$, where $\rho_i$ denotes the border strip consisting of those boxes labelled by $i$ in the tableaux below.
$$ \young(1223,1233,1444)\ ,  \quad \young(5677,5678,5688)\ . $$ 
Since $\epsilon_3(4,4,4)=1$, Theorem~\ref{thm:hookdefl} implies that $(\Defres^{\chi^{(3,1)}}_{S_3} \chi^{(4,4,4)})(g)=2$.
\end{example}

The proof of Theorem~\ref{thm:hookdefl} requires the following definition, 
in which we assume that $a+b=m$. 

\begin{definition}\label{def:D_a,b}
Let $\kappa$ be a skew shape of size $mn$. Define $\mathcal{D}^\kappa_{a,b}$ to be the set of pairs 
$(D,E)$ where $D$ and $E$ are border-strip $n$-diagrams such that $\kappa=\bigcup (D\cup E)$ and 
 the following conditions are satisfied:       
\begin{enumerate} 
 \item[($1'$)] $|D| = a$ and $|E| = b$; 
 \item[(3)] $D$ is a horizontal border-strip $n$-diagram, and $E$ is a vertical border-strip $n$-diagram;
 \item[(4)] there do not exist $\rho_D\in D$ and $\rho_E \in E$ such that $(\rho_E, \rho_D) \in \mathcal{R}$. 
\end{enumerate}
\end{definition}

Note that since $|\kappa| = mn$, the border-strips in $D$ and $E$ are necessarily disjoint.
We  need the following two lemmas on $\mathcal{D}^\kappa_{a,b}$; their proofs are given after the proof of
Theorem~\ref{thm:hookdefl}.

\begin{lemma}\label{lemma:D_a,b}
Let $\kappa$ be a skew shape of size~$mn$, and let $a,b$ be positive integers such that $a+b=m$. Then
\[
\lan \omega(\Res_{S_m \wr S_n} \chi^\kappa), \ \Ind_{S_a \times S_b}^{S_m} (\chi^{(a)} \times \chi^{(1^b)}) \ran = \epsilon_n(\kappa) | \mathcal{D}^\kappa_{a,b}|,
\]
\end{lemma}

\begin{lemma}\label{lemma:DandB}
Let $\kappa$ be a skew shape of size~$mn$. Let $0\le a <m $ and $b=m-a$. Then
\[
 |\mathcal{D}^{\kappa}_{a,b}| = |\mathcal{B}^{\kappa}_{a,b}| + |\mathcal{B}^{\kappa}_{a+1,b-1}|.
\]
\end{lemma}

\begin{proof}[Proof of Theorem~\ref{thm:hookdefl}]
The proof is by induction on $b$. Recall that $g$ denotes an $n$-cycle.
In the case $b=0$ we have
\[ (\Defres^{\chi^{(m)}}_{S_n} \chi^{\kappa})(g)= a_{\kappa, (n)}= \epsilon_n(\kappa) | \mathcal{D}^\kappa_{m,0}|=\epsilon_n(\kappa) | \mathcal{B}^\kappa_{m,0}|, \]
where the first equality follows from Theorem~\ref{thm:cd} and  the second by 
Lemma~\ref{lemma:horiz}.

For $b>0$, we combine Proposition~\ref{prop:omega} with Lemmas~\ref{lemma:D_a,b} and~\ref{lemma:DandB}, Young's Rule and the inductive hypothesis to get
\begin{align*}
(\Defres^{\chi^{(a ,1^{b})}}_{S_n} \chi^{\kappa}) (g) &=
 \lan \omega (\Res_{S_m \wr S_n} \chi^{\kappa}), \chi^{(a,1^b)} \ran \\
 & =    \lan \omega (\Res_{S_m \wr S_n} \chi^{\kappa}), \Ind_{S_a \times S_b}^{S_m} (\chi^{(a)} \times \chi^{(1^b)}) - \chi^{(a+1,1^{b-1})} \ran\\
& =  \epsilon_n(\kappa) |\mathcal{D}^\kappa_{a,b}|- (\Defres^{\chi^{(a+1 ,1^{b-1})}}_{S_n} \chi^{\theta}) (g)\\
& =  \epsilon_n(\kappa) |\mathcal{D}^\kappa_{a,b}| -  \epsilon_{n} (\kappa) |\mathcal{B}^{\kappa}_{a+1,b-1}|\\
 & =   \epsilon_{n} (\kappa) |\mathcal{B}^{\kappa}_{a ,b}|.
\end{align*}
as required.
\end{proof}

It remains to demonstrate the truth of the two lemmas.

\begin{proof}[Proof of Lemma~\ref{lemma:D_a,b}]
Let $g$ denote an $n$-cycle. In the case $b=0$, 
\[
\lan \omega(\Res_{S_m \wr S_n} \chi^\kappa), \ \chi^{(m)} \ran = (\Defres _{S_n} \chi^{\kappa}) (g)=a_{\kappa, (n)}=\epsilon_n(\kappa)|\mathcal{D}^{\kappa}_{m ,0}|
\]
by Proposition~\ref{prop:omega}, Theorem~\ref{thm:cd} and Lemma~\ref{lemma:horiz}.
Similarly, the case $a=0$ follows using Proposition~\ref{prop:sgn}:
\[
\lan \omega(\Res_{S_m \wr S_n} \chi^\kappa), \ \chi^{(1^m)} \ran = (\Defres^{\sgn_{S_m}} _{S_n} \chi^{\kappa}) (g) =\epsilon_n(\kappa)|\mathcal{D}^{\kappa}_{0 ,m}|.
\]

In the general case, let $\kappa=\lambda / \mu$. By Frobenius reciprocity and Lemma~\ref{lemma:skewres},
\[ \begin{split}
\lan \omega(\Res_{S_m \wr S_n} \chi^\kappa), \ \Ind_{S_a \times S_b}^{S_m} (\chi^{(a)} \times \chi^{(1^b)}) \ran 
= \hskip2in \\ 
\hskip1in \lan  \sum_\tau \omega(\Res_{S_a \wr S_n} \chi^{\tau / \mu}) \times \omega(\Res_{S_b \wr S_n} \chi^{\lambda / \tau}), \chi^{(a)} \times \chi^{(1^b)} \rangle, 
\end{split} \] 
where the sum is over all partitions $\tau$ such that
 $\mu \subseteq \tau \subseteq \lambda$ and $|\tau / \mu| = an$.
Using the two extreme cases, it follows that
\begin{align*}
\lan \omega(\Res_{S_m \wr S_n} \chi^\kappa), \ \Ind_{S_a \times S_b}^{S_m} (\chi^{(a)} \times \chi^{(1^b)}) \ran &=   \sum_\tau \epsilon_n(\tau / \mu) |\mathcal D^{\tau / \mu}_{a,0}| \epsilon_n(\lambda / \tau) |\mathcal D^{\lambda / \tau}_{0,b}|\\
 & = \epsilon_n(\lambda / \mu) \sum_\tau |\mathcal D^{\tau / \mu}_{a,0} | \, |\mathcal D^{\lambda / \tau}_{0,b}|\\
&=   \epsilon_n(\lambda / \mu) |\mathcal D^{\lambda/ \mu}_{a,b}|
\end{align*}
as required.
\end{proof}

The proof of Lemma~\ref{lemma:DandB} relies upon the following simple result.

\begin{lemma}\label{lemma:compar}
 Let $\rho_1$ and $\rho_2$ be disjoint border strips, each of length~$n$, such that $(\rho_1, \rho_2) \in \mathcal{R}$. 
For $t\in \{1,2\}$, let $(i_t,j_t)$ and $(k_t,\ell_t)$ be  the initial and terminal boxes  respectively of $\rho_t$.
Then 
\begin{enumerate}
 \item There exists $(r,s)\in \rho_2$ such that
either $r \ge i_1$ and $s> j_1$, or $r> k_1$ and $s\ge \ell_1$.  
\item There exists $(t,q)\in \rho_1$ such that either
$t< i_2$ and $q\le j_2$, or $t \le k_2$ and $q<\ell_2$. 
\end{enumerate}
\end{lemma}

\begin{proof}
 We   prove the first statement only; the second is entirely analogous.
Suppose that $(\rho_1, \rho_2) \in \mathcal{R}$ but there exists no  $(r,s)\in \rho_2$   satisfying the stated conditions.
Since $(\rho_1, \rho_2)\in  \mathcal{R}$, there exist $(a,b ) \in \rho_1$ and $(e,f)\in \rho_2$ such that $(a,b) <_p (e,f)$. In particular, $i_1 \le e$ and $\ell_1\le f$, and our assumption implies that $f\le j_1$ and $e\le k_1$.  
Thus $(e,f)$ belongs to the rectangle $[i_1,k_1]\times [\ell_1,j_1]$.
The ribbon $\rho_1$ divides its complement in $[i_1,k_1]\times [\ell_1,j_1]$ into two connected components, with $(e,f)$ lying to the south east of $\rho_1$ (as $(a,b) \in \rho_1$ satisfies $(a,b)  <_p (e,f)$).

Let $\tau = \{ (c, j_1+1) : i_1\le c\le k_1+1\} \cup \{ (k_1+1, d) : \ell_1\le d\le j_1+1 \}$, and observe that our assumption ensures that $\tau \cap \rho_2=\varnothing$.  The set $\rho_1 \cup \tau$ is the boundary of a certain region $\Delta$. Since $(e,f)\in \Delta$ and $\rho_2$ does not intersect the boundary, the whole border strip $\rho_2$ must be contained in $\Delta$ and, in particular,   $\rho_2 \subset [i_1,k_1]\times [\ell_1,j_1]$.
However, as $\rho_1$ is a border strip of the same length with initial and terminal boxes  $(i_1,k_1)$ and $(\ell_1,j_1)$ respectively,  $\rho_2$ must
contain the corners $(i_1,k_1)$ and $(\ell_1,j_1)$ and hence intersect $\rho_1$, contrary to our hypotheses.
\end{proof}

This lemma can be used to verify that if $D$ is a horizontal border-strip $n$-diagram with initial box $(i_\sigma, j_\sigma)\in \sigma \in D$ then $(\sigma, \rho) \notin \mathcal{R}$ for all $\rho \in D$.
Indeed, if 
$(\sigma, \rho) \in \mathcal{R}$ then by Lemma~\ref{lemma:compar}(1) 
there exists $(r,s) \in \rho$ such that
either $r \ge i_\sigma$ and $s > j_\sigma$, or
$r>k_\sigma$ and $s \ge \ell_\sigma$.
The `either' case is impossible because  
$(i_\rho, j_\rho)$ is the initial box of $D$.
Hence $(i_\rho, j_\rho)$ must lie in  $[i_\sigma,r]\times [s,j_\sigma]$, to the south of $\sigma$, and so a box of $\sigma$ lies above   $(i_\rho, j_\rho)$, contradicting the horizontality of $D$.
Similarly,  if $E$ is a vertical border-strip $n$-diagram 
with initial box $(i_{\sigma}, j_{\sigma}) \in \sigma \in E$
 then $(\rho,\sigma) \notin \mathcal{R}$ for all $\rho \in E$: indeed if 
$(\rho,\sigma) \in \mathcal{R}$ then by Lemma~\ref{lemma:compar}(2) 
there exists $(t,q) \in \rho$ such that
either $t < i_\sigma$ and $q \le j_\sigma$, or
$t\le k_\sigma$ and $q < \ell_\sigma$. The `either' case
is again ruled out because $(i_{\sigma}, j_{\sigma})$ is the initial box of $E$. 
Hence 
$$(t,q) <_p (k_\sigma,q) \le_p (k_\sigma, \ell_\sigma),$$ 
and so $(k_{\sigma}, q)$ is a box of $E$ because $\bigcup E$ is convex. 
But this box lies to the left of the terminal box of $\sigma$, a contradiction.
We shall use these observations in the proof of Lemma~\ref{lemma:DandB}.

\begin{proof}[Proof of Lemma~\ref{lemma:DandB}]
We construct a bijection
$$f \colon \mathcal{B}^{\kappa}_{a,b} \sqcup \mathcal{B}^{\kappa}_{a+1,b-1}  
\to \mathcal{D}^{\kappa}_{a,b}.$$
Given $(D,E) \in \mathcal{B}^{\kappa}_{a,b} \sqcup \mathcal{B}^{\kappa}_{a+1,b-1}$, let $\sigma$ denote the unique element of $D\cap E$. For $(D,E) \in \mathcal{B}^{\kappa}_{a,b}$, we set 
$f(D,E)=(D,E \setminus \{\sigma\})$, and
for $(D,E) \in \mathcal{B}^{\kappa}_{a+1,b-1}$ we set 
$f(D,E)=(D\setminus \{\sigma\}, E)$.

Firstly, we verify that $f(D,E) \in \mathcal{D}^{\kappa}_{a,b}$. Suppose   
that $(D,E) \in \mathcal{B}^{\kappa}_{a,b}$ (as the second case is exactly analogous). 
The conditions ($1'$), (3), (4) 
on \hbox{$(D,E \setminus \{\sigma\})$} 
follow immediately
provided that $\bigcup\, (E \setminus \{\sigma\})$ is a skew shape.  Take $x<_p y<_p z$ with 
$x,z \in \bigcup\, (E \setminus \{\sigma\})$. We know $y$ lies in the skew shape $\bigcup E$, so 
suppose that $y \in \sigma$. Then if $z \in \rho_E \in E \setminus \{\sigma \}$
we have $(\sigma, \rho_E) \in \mathcal{R}$, contrary to 
condition (4) of Definition~\ref{def:hook-like}.

To see that $f$ is the bijection we require, we define its inverse map			
$$h \colon   \mathcal{D}^{\kappa}_{a,b} \to  \mathcal{B}^{\kappa}_{a,b} \sqcup \mathcal{B}^{\kappa}_{a+1,b-1}.$$
For $(D,E) \in \mathcal{D}^{\kappa}_{a,b}$, let $\sigma$ denote the border strip of $D \cup E$ containing the initial box of $\kappa$. Then if $\sigma \in D$ we set $h(D,E)=(D, E \cup \{\sigma\})$ and if $\sigma \in E$ we set $h(D,E)=(D \cup \{\sigma\}, E)$.

We verify that $h(D,E) \in  \mathcal{B}^{\kappa}_{a,b} \sqcup \mathcal{B}^{\kappa}_{a+1,b-1}$. 
Suppose $\sigma \in D$.  To see that $\bigcup\,(E \cup \{\sigma\})$ is a skew shape, take $x <_p y<_p z$ with $x ,z \in\bigcup\,( E \cup \{\sigma\})$. If $x,z$ lie in the skew shape $\bigcup E$ then so does $y$, and similarly if  $x,z \in \sigma$ then $y \in \sigma$. Two cases remain. Firstly, if $x \in \rho \in E$ and $z \in \sigma \in D$ then $(\rho, \sigma) \in \mathcal{R}$, contrary to the definition of $ \mathcal{D}^{\kappa}_{a,b}$. 
Secondly, suppose that $x \in \sigma \in D$ and $z   \in \bigcup E$,
and, for a contradiction, that
$y \in \rho \in D \setminus\{ \sigma\}$. Then $(\sigma, \rho) \in \mathcal{R}$, which 
is impossible by the first observation following Lemma~\ref{lemma:compar}.
(The proof that $D \cup \{\sigma\}$ is a skew shape
in the case $\sigma \in E$  is entirely analogous.)

Next we verify that if $\sigma \in D$ then $E \cup \{\sigma \}$ is a vertical $n$-border strip diagram. Let $(i_\sigma, j_\sigma)$ be the initial box of $\sigma$
and let $(k_\sigma, \ell_\sigma) \in \sigma$ be the terminal box of $\sigma$.
Since $E$ is a vertical $n$-border strip diagram,
it suffices to check firstly that there is no box $(k_\sigma, \ell) \in \rho_E \in E$ for $\ell <\ell_\sigma$, and secondly that, if $\rho \in E$ has terminal box $(k_\rho, \ell_\rho)$,
then $(k_\rho, j) \notin \sigma$ for any $j<\ell_\rho$.
The first statement is a consequence of the definition of  $ \mathcal{D}^{\kappa}_{a,b}$ since the existence of such a box implies $(\rho_E,\sigma) \in \mathcal{R}$.  
If the second statement fails 
with $(k_\rho, j) \in \sigma$, then $(\sigma, \rho) \in \mathcal{R}$ and 
Lemma~\ref{lemma:compar}(1) 
implies that there is a box $(r,s) \in \rho$ with either $r \ge i_\sigma$ and $s > j_\sigma$, or
$r>k_\sigma$ and $s \ge \ell_\sigma$. The `either' case is ruled out
because $(i_\sigma, j_\sigma)$ is the initial box of $\sigma$. 
In the `or' case, since  $(k_\rho, j) \in \sigma$, we have
$ k_\sigma\ge k_\rho$, and therefore $r> k_\sigma\ge k_\rho$, contradicting that $(r,s) \in \rho$. (The proof that if  $\sigma \in E$ then $D \cup \{\sigma \}$ is a horizontal $n$-border strip diagram is  simpler, 
using the definition of  $ \mathcal{D}^{\kappa}_{a,b}$  and the fact that $(i_\sigma, j_\sigma)$ is the initial box of $\theta$.)

Finally, we must check that if $\sigma \in D$ then there do not exist $\rho \in D$ and $\rho' \in E \cup \{\sigma \}$ with $(\rho', \rho) \in \mathcal{R}$.
This is true because $(D,E) \in \mathcal{D}^{\kappa}_{a,b}$,  and, by the observation following Lemma~\ref{lemma:compar}, for all $\rho \in D$,   $(\sigma, \rho) \notin \mathcal{R}$.
(Again, the case $\sigma \in E$ is similar).  

We have demonstrated that $f$ and $h$ are well-defined, and by their construction the maps are 
mutually inverse.
\end{proof}

This completes the proof of Theorem~\ref{thm:hookdefl} on deflation with respect to hook characters.
We now give some results on 
$\Defres^\theta_{S_n} \chi^{\lambda}$ for  an  arbitrary irreducible character~$\theta$.

 Proposition~\ref{prop:gen_basecase} combined with the Littlewood--Richardson rule 
 yields the following corollary, which
gives a useful sufficient condition for the deflation of an irreducible character of $S_{mn}$ to vanish on an $n$-cycle.

\begin{corollary} \label{cor:nv}
Let $m$, $n \in \N$, let $\lambda$ be a partition of $mn$.
Let $\beta$ be a partition of $m$ and let $\theta = \chi^\beta$.
Let $g \in S_n$ be an $n$-cycle.
If $(\Defres^\theta_{S_n} \chi^{\lambda})(g) \not= 0$ 
then $\lambda$ has empty $n$-core and moreover,
$\lambda^{(i)} \subseteq \beta$ for each $i \in \{0,\ldots,n-1\}$,
where
 $(\lambda^{(0)}, \ldots, \lambda^{(n-1)})$ is the $n$-quotient of $\lambda$. 
In this case  
$$(\Defres^\theta_{S_n} \chi^{\lambda})(g) = \nsgn_n(\lambda / \emptyset) c^\beta_{\lambda^{(0)}\ldots\lambda^{(n-1)}}$$
where  $c^\beta_{\lambda^{(0)}\ldots\lambda^{(n-1)}}$ denotes a generalized Littlewood--Richardson coefficient.
\end{corollary}

A related result is Proposition~\ref{prop:gen_deg} below, which 
gives the degrees of the deflations of the irreducible
characters of $S_{mn}$ to $S_n$. It may be proved in the
same way as Equation~\eqref{eq:deflation_special} in \S 1. 
Note that the right-hand side equals the generalized Littlewood--Richardson coefficient
$c^\lambda_{\beta\ldots\beta}$.

\begin{proposition}\label{prop:gen_deg} 
Let $m$, $n \in \N$.
Let $\beta$ be a partition of $m$ and let $\theta = \chi^\beta$. Then
\[
\Defres^{\theta}_{S_n}(\chi^\lambda)(1_{S_n}) = 
\left<
\Ind^{S_{mn}}_{S_m \times \cdots \times S_m} \chi^{\beta} \times \cdots \times \chi^\beta, 
 \chi^\lambda
\right> \]
for any partition $\lambda$ of $mn$.\hfill$\qed$
\end{proposition}

We end with an
example showing how Propositions~\ref{prop:gen_basecase},~\ref{prop:defres_decomp}, and ~\ref{prop:gen_deg} and Corollary~\ref{cor:nv} may be
used to calculate the values of an irreducible character of~$S_{mn}$ deflated
with respect to an arbitrary character of $S_m$.

\begin{example}
Let $\theta = \chi^{(2,2)}$. We shall find $(\Defres^{\theta}_{S_4}  \chi^{(6,4,4,2)})(g)$ in the cases where $g\in S_4$ is a transposition or a double transposition. 

Firstly take $g$ to be a transposition.  
By Proposition~\ref{prop:defres_decomp} we have
 \[ (\Defres^{\theta}_{S_4} \!\chi^{(6,4,4,2)})(g) = \sum_{\tau} (\Defres^{\theta}_{S_2} \chi^{\tau})(1_{S_2})  (\Defres^{\theta}_{S_2} \chi^{(6,4,4,2) / \tau})(k) \] where $k = (12) \in S_2$ and the sum is over all partitions $\tau$ of $8$. 
Using  Proposition~\ref{prop:gen_deg} on the first term we obtain 
\[ (\Defres^{\theta}_{S_4} \!\chi^{(6,4,4,2)})(g) = \sum_{\tau} c^{\tau}_{(2,2)(2,2)} (\Defres^{\theta}_{S_2} \chi^{(6,4,4,2) / \tau})(k) \] with the same conditions on the sum. 
 By  Proposition~\ref{prop:gen_basecase}, we need only consider those partitions $\tau$ such that $(6,4,4,2) / \tau$ is 2-decomposable.
 The $2$-quotient of $(6,4,4,2)$ is $\bigl((2,1), (3,2)\bigr)$, so the $\tau$ we must consider
are the partitions   $(6,1^2)$, $(4,3,1)$, $(6,2)$,  $(4,2^2)$, $(4^2)$,  $(3^2,2)$,  $(4,2,1^2)$,  $(2^4)$,  $(3^2,1^2)$. Calculation shows that $c^{\tau}_{(2,2)(2,2)} = 1$ 
when $\tau \in P$ where
\[ P=\{
(4,3,1),(4,2^2), (4^2), (2^4), (3^2,1^2) \} \] 
and that $c^{\tau}_{(2,2)(2,2)}$ is zero in the other four cases. 
Hence by  Proposition~\ref{prop:gen_basecase} we have
 \[ \begin{split}
 (\Defres^{\theta}_{S_4} \chi^{(6,4,4,2)})(g) = \hskip3in \\ \hskip1in \sum_{\tau \in P} \nsgn_2\bigl((6,4,4,2) / \tau\bigr) \left< \Ind^{S_4}_H   \chi^{(2,1) / \tau^{(0)}} \times \chi^{(3,2) / \tau^{(1)}}, \chi^{(2,2)} \right> \end{split} \] 
where    $H=S_{|(2,1) / \tau^{(0)}|} \times S_{|(3,2) / \tau^{(1)}|}$. 
The contributions to the sum from the elements of $P$,  in the order given above, are $-1$, $+2$, $+1$, $+1$ and $+1$ respectively. For example, the $2$-quotient of $(6,4,4,2) / (4,2,2)$ is 
\[ \bigl((2,1) / (1), (3,2) / (2,1)\bigr) \]  
and $\nsgn_2\bigl( (6,4,4,2) / (4,2,2) ) = 1$, so the contribution from $(4,2,2)$ is 
\[ \left< \Ind^{S_4}_{S_2 \times S_2} \chi^{(2,1) / (1)} \times \chi^{(3,2) / (2,1)}, \chi^{(2,2)} \right> = 
\left< \Ind_{1}^{S_4} 1, \chi^{(2,2)} \right> = 2. \] 
Therefore $(\Defres^{\theta}_{S_4} \chi^{(6,4,4,2)})(g) = 4$.

Similar arguments can be used in the case where $g$ is a double transposition.  By Proposition~\ref{prop:defres_decomp} we have
 \[ (\Defres^{\theta}_{S_4} \!\chi^{(6,4,4,2)})(g) = \sum_{\tau} (\Defres^{\theta}_{S_2} \chi^{\tau})(k)  (\Defres^{\theta}_{S_2} \chi^{(6,4,4,2) / \tau})(h) \]
 where $k$ and $h$ are transpositions, and the sum is over all partitions $\tau$ of~$8$.  Proposition~\ref{prop:gen_basecase} and   Corollary~\ref{cor:nv} restrict the possible partitions $\tau$ to be considered and show that
 $(\Defres^{\theta}_{S_4} \chi^{(6,4,4,2)})(g)$ equals
 \[  \sum_{\tau \in P} \nsgn_2\bigl(\tau / \emptyset \bigr) c^{(2,2)}_{\tau^{(0)} \tau^{(1)}} \nsgn_2\bigl((6,4,4,2) / \tau\bigr)
 \left< \Ind^{S_4}_H   \chi^{(2,1) / \tau^{(0)}} \times \chi^{(3,2) / \tau^{(1)}}, \chi^{(2,2)} \right> \] where  
 \[ P = \{ (4,3,1), (4,2^2), (4^2), (2^4), (3^2,1^2) \} \] 
 and  $H=S_{|(2,1) / \tau^{(0)}|} \times S_{|(3,2) / \tau^{(1)}|}$. The contributions to the sum from the elements of $P$, in the order given above, are $+1$, $+2$, $+1$, $+1$ and $+1$ respectively. 
Hence
$(\Defres^{\theta}_{S_4} \!\chi^{(6,4,4,2)})(g) = 6$.  
 \end{example}

\section{Symmetric functions}\label{sec:concl} 

Finally, we discuss the translation of our results
into the language of symmetric functions.
The following well-known facts can be found in~\cite[Chapter I]{Macdonald1995}.
Let $\Lambda$ be the ring of symmetric functions with integer 
coefficients in variables $x_1, x_2,\ldots$, and let $R=\bigoplus_{n\ge 0} \, \Ch(S_n)$. There is a well-known canonical ring isomorphism $\ch\colon R\to \Lambda$, 
where the ring structure on $R$ is given, for $f \in \mathcal{C}(S_m)$ and $g\in \mathcal{C}(S_n)$, by  
$fg = \Ind_{S_m \times S_n}^{S_{mn}}(f \times g)$.
Moreover, the map $\ch$ is an isometry with respect to the standard 
inner products $\lan \cdot, \hskip1pt \cdot \ran$ on $R$ and $\Lambda$. If $\lambda/\mu$ is any skew-partition, 
then $\ch (\chi^{\lambda/\mu}) = s_{\lambda/\mu}$, the skew Schur function corresponding to $\lambda/\mu$. 
Furthermore, suppose that $\beta$ and $\nu$ are partitions, with $|\beta|=m$ and $|\nu|=n$. 
Denote by $s_{\nu}\circ s_{\beta}$ the plethysm of $s_{\nu}$ and $s_{\beta}$ (see~\cite[\S I.8]{Macdonald1995}). 
Then 
\begin{equation}\label{eq:ch}
 \ch\left( \Ind_{S_m\wr S_n}^{S_{mn}} (\widetilde{(\chi^{\beta})^{\times n}} \Inf_{S_n}^{S_m\wr S_n} \chi^{\nu}) \right) = s_{\nu} \circ s_{\beta}
\end{equation}
(see~\cite[\S I.8 and \S I.A.6]{Macdonald1995}). 
By first using Frobenius reciprocity, then the inflation-deflation reciprocity relation of Equation~(\ref{eq:reciprocity}), we have
\[ \left<
\Ind_{S_m\wr S_n}^{S_{mn}} (\widetilde{(\chi^{\beta})^{\times n}} \Inf_{S_n}^{S_m\wr S_n} \chi^{\nu}),
\chi^{\lambda/\mu}
\right> = \left< \chi^\nu, \Defres^{\chi^{\beta}}_{S_n} \chi^{\lda/\mu} \right> \]
for any skew-partition $\lambda / \mu$ with $|\lambda / \mu| = mn$. Comparing with Equation~\eqref{eq:ch} shows
that if $\lambda/\mu$ is such a skew-partition then
\begin{equation}\label{eq:defpl}
 \Defres^{\chi^{\beta}}_{S_n} \chi^{\lda/\mu} = \sum_{\nu} \lan s_{\lda/\mu}, s_{\nu}\circ s_{\beta} \ran \chi^{\nu}
\end{equation}
where the sum is over all partitions $\nu$ of $n$. 

Let $p_l = \sum_{i} x_i^l$ be the power-sum symmetric function, and write $p_{\gamma} = p_{\gamma_1} \cdots p_{\gamma_d}$ for any composition $\gamma=(\gamma_1,\ldots, \gamma_d)$.
If $\gamma$ is a composition of $n$, then $\chi^{\nu} (g_{\gamma}) = \lan p_{\gamma}, s_{\nu} \ran$ (see~\cite[Equation (I.7.8)]{Macdonald1995}).
Using Equation~\eqref{eq:defpl}, we obtain
\begin{equation}\label{eq:trans}
 \begin{split}
  (\Defres_{S_n}^{\chi^{\beta}} \chi^{\lda/\mu})(g_{\gamma}) &= \sum_{\nu} \lan s_{\lda/\mu}, s_{\nu} \circ s_{\beta} \ran \chi^{\nu} (g_{\gamma}) \\
 &= \sum_{\nu}\lan s_{\lda/\mu}, s_{\nu} \circ s_{\beta} \ran \lan s_{\nu}, p_{\gamma} \ran = \lan  s_{\lda/\mu}, p_{\gamma} \circ s_{\beta} \ran
 \end{split}
\end{equation}
where the sums are over all partitions of $n$.

In the case when $\beta = (m)$, we have $s_{\beta} = h_m$, where $h_m$ is the complete symmetric function of degree $m$ (see~\cite[Section I.2]{Macdonald1995}). 
Thus Equation~\eqref{eq:trans} shows that Theorem~\ref{thm:cd} is equivalent to the identity
\begin{equation}\label{eq:cd2}
 \lan  s_{\lda/\mu}, p_{\gamma} \circ h_m \ran = a_{\lda/\mu, \gamma}.
\end{equation}

In~\cite[Section 9]{DLT}, D{\'e}sarm{\'e}nien, Leclerc, and Thibon obtain a formula which implies that 
\begin{equation}\label{eq:DLT}
 s_{\mu} (p_n \circ h_m) = \sum_{\lambda} \epsilon_n (\lambda/\mu) a_{\lambda/\mu, (n)},
\end{equation}
where the sum is over the partitions $\lambda$ of $|\mu|+mn$ such that $\lambda\supseteq \mu$. 
(Our definition of an $m$-border-strip tableau of type $(n)$ is equivalent to the definition of a horizontal $n$-ribbon tableau of weight $m$ in~\cite{DLT}.)
In fact, the formula in~\cite{DLT} is more general, giving a combinatorial description of $s_{\mu} (p_n \circ s_{\kappa})$ for any partition $\kappa$ of $n$.
Clearly, Equation~\eqref{eq:DLT} is equivalent to Equation~\eqref{eq:cd2} in the case $\gamma=(n)$; this leads, after some work, to an alternative proof of Theorem~\ref{thm:cd}. 

Furthermore, in the special case when $\mu=\emptyset$, 
a combinatorial description of the left-hand side of Equation~\eqref{eq:cd2} is given by Macdonald in~\cite[\S I.8, Example 8]{Macdonald1995}. 
Using Lemma~\ref{lemma:horiz}, one can see that
this description is equivalent to the
definition of $a_{\lda, \gamma}$. 

\section*{Acknowledgements}
The authors would like to thank Prof.~Christine Bessenrodt for supplying the
reference to Farahat's paper \cite{Farahat1954}.

\def\cprime{$'$} \def\Dbar{\leavevmode\lower.6ex\hbox to 0pt{\hskip-.23ex
  \accent"16\hss}D}

\end{document}